\documentclass{amsart}

\usepackage{amsmath}
\usepackage{amsfonts}
\usepackage{amstext}
\usepackage{amsbsy}
\usepackage{amsopn}
\usepackage{amsxtra}
\usepackage{upref}
\usepackage{amsthm}
\usepackage{amsmath}
\usepackage{amssymb}
\usepackage{epsfig,verbatim,ifthen,graphicx}
\usepackage[all]{xy}
\newtheorem{prop}{Proposition}[section]
\newtheorem{lema}{Lemma}[section]
\newtheorem{teo}{Theorem}[section]

\newtheorem{coro}{Corollary}[section]
\newtheorem{claim}{Claim}[section]

\theoremstyle{definition}
\newtheorem{defi}{Definition}[section]
\newtheorem{rem}{Remark}[section]

\def\R{{\mathbb R}}

\def\N{{\mathbb N}}

\def\F{{\mathcal F}}
\def\M{{\mathcal M}}
\def\K{{\mathcal K}}

\title[Almost-additive thermodynamic formalism ]{Almost-additive thermodynamic formalism for countable Markov shifts}
\date{\today}

\begin{thanks}
{GI was partially  supported by Proyecto Fondecyt 11070050. YY was supported by Proyecto Fondecyt Postdoctorado 3090015 and Basal grant. }
\end{thanks}

\author{Godofredo Iommi} \address{Facultad de Matem\'aticas,
Pontificia Universidad Cat\'olica de Chile (PUC), Avenida Vicu\~na Mackenna 4860, Santiago, Chile}
\email{giommi@mat.puc.cl}
\urladdr{http://www.mat.puc.cl/\textasciitilde giommi/}
\author{Yuki Yayama} \address{Centro de Modelamiento Matem\'{a}tico, Universidad de Chile, Avenida Blanco Encalada 2120, Santiago, Chile}
\email{yyayama@dim.uchile.cl}

\begin{document}

\begin{abstract}
We introduce a definition of pressure for almost-additive sequences of continuous functions defined over (non-compact) countable Markov shifts. The variational principle 
is proved. Under certain assumptions we prove the existence of Gibbs and equilibrium measures.  Applications are given to the study of maximal Lypaunov exponents of product of matrices.  
\end{abstract}

\maketitle

\section{Introduction}
This paper has two different starting points. On the one hand, we have the thermodynamic formalism developed 
for sub-additive and almost-additive sequences of continuous functions.
The idea of this theory is to generalise classical results on thermodynamic formalism replacing the 
pressure of a continuous function with the pressure of a sequence of  continuous functions. It 
was  Falconer \cite{f} who introduced  this set of ideas  with the purpose of studying dimension 
theory of non-conformal systems. Recall that the relation between thermodynamic formalism and 
dimension theory has been extensively and successfully exploited ever since the pioneering work of 
Bowen \cite{bo2} (see the books \cite{b4, f2, Pe,Pru} for recent developments of the theory).  If 
an expanding  dynamical system, $T:M \to M$,  is conformal then it is possible to describe 
in great detail the Hausdorff dimension of dynamically defined subsets of the phase space. 
For instance, making use of the classic thermodynamic formalism, it is possible to study 
the size (e.g. Hausdorff dimension) of level sets determined by pointwise dimension of Gibbs 
measures or by Lyapunov exponents (see \cite[Chapter II and III]{b4}). The situation is far 
less developed if no conformal  assumption is made on the system. There are several reasons for 
this, one of them being that, in the conformal setting, dynamically defined balls are \emph{almost} balls and they form an optimal cover. However, in the non-conformal setting
 dynamically defined balls are  ellipses. Therefore, it is  likely that the dynamical cover is not optimal.  This is clearly related to the fact that the natural function used to estimate the dimension, namely the Jacobian,  in the conformal setting satisfies 
\[ \Vert DT ^{m+n} (x)\Vert = \Vert DT^m (T^n x) \Vert   \Vert DT^n (x) \Vert,\]
where $T^n$ denotes the $n-$th iterate of the map $T$ and  $\Vert \cdot \Vert$ is the operator norm. Whereas, if the map is not conformal we only have
\[ \Vert DT ^{m+n} (x)\Vert \leq \Vert DT^m (T^n x) \Vert   \Vert DT^n (x) \Vert.\]
Therefore, the sequence defined by $\phi_n(x) = \log \Vert DT^{n}(x)\Vert$
is sub-additive (not additive) and a new thermodynamic formalism is required to deal with this situation. This was the main motivation of Falconer \cite{f}. This problem has attracted a great deal of attention over the last two decades and recently interesting developments have been obtained.   We would like to single out the work of Barreira \cite{b1,b2,b3}  among many other substantial contributions to the theory. We will build up on his work.

Our second starting point is the ergodic theory for countable Markov shifts. Uniformly hyperbolic dynamical systems have finite Markov partitions, see \cite[Chapter 3]{bo3} and references therein for the case of discrete time dynamical systems. That type of coding allows for the proof of a great deal of fundamental results in ergodic theory. When systems are uniformly hyperbolic in most of the phase space but not in all of it, sometimes it is still possible to construct Markov partitions, although this time over countable alphabets. Probably, the best known example of such a situation being the Manneville-Pomeau map \cite{mp,s2}, which is an expanding interval map with a parabolic fixed point. Dynamical systems that can be coded using countable Markov partitions also occur in the study of one-dimensional real and complex multimodal maps. Indeed, a successful  technique used to study the ergodic theory of these maps and overcome the lack of hyperbolicity due to the existence of critical points is the so called \emph{inducing procedure}  (see  for example    \cite{bt2,it1, ps, pr}). Given a multimodal map it is possible  to associate  an induced map (which is a generalisation of the first return time map) which posses a countable Markov partition. The idea  is to translate problems to this new system, solve them and then push the results back.  The case of $C^r$ diffeomorphisms defined over compact orientable smooth surface  was recently studied by Sairg \cite{s5}. He constructed countable Markov partitions for large invariant sets. While all the above examples are important, the most natural ones arise in number theory. Indeed, the Gauss map and the map associated to the Jacobi-Perron algorithm are both  Markov over countable partitions (see \cite{ma, m2, pw}). The thermodynamic formalism for countable Markov shifts has been developed by Mauldin and  Urba\'nski \cite{mu1, mu} and by Sarig \cite{s1,s2,s3,s4} (see also \cite{ffy,gs}). The main difficulty is that the phase space is no longer compact, therefore fixed point theorems used in the compact setting can not be applied here and new techniques have to be developed.

The aim of this paper is to bring these two theories together. In doing so, we develop a set of 
tools that could be used to tackle a wide range of problems for which, at present time,  no 
machinery was available. Generalising the work of Barreira \cite{b1,b2,b3}, of Mauldin and Urba\'nski \cite{mu1, mu} and 
that of Sarig \cite{s1,s2,s3,s4}, we define a notion of pressure for almost-additive sequences of functions 
defined over a (non-compact) countable Markov shift (see Section \ref{sec:def}).  We prove that this pressure satisfies the variational principle (see Section \ref{sec:var}). The problem of the existence of Gibbs measures is also addressed, in Section \ref{sg}, under a combinatorial assumption  on the shift (that of being BIP) we prove the existence of Gibbs measures. We stress that our definition  and the variational principle hold for \emph{any} topologically mixing countable Markov shift.  
%

As an application of our results, we study the Maximal Lyapunov exponents of product of 
matrices (see Section \ref{le}). We consider  a countable collection of $d \times d$ 
matrices, $\{A_1, A_2, \dots \}$,  and   a topologically mixing countable Markov 
shift $(\Sigma, \sigma)$. If $w=(i_0, i_1, \dots) \in \Sigma$  define a sequence of 
functions by
\begin{equation*}
\phi_n(w)= \log \Vert A_{i_{n-1}} \cdots  A_{i_1} A_{i_0} \Vert.
\end{equation*}
The thermodynamic formalism for such class of sub-additive sequence has been extensively 
studied over the last years. Feng, in a series of articles  \cite{Fe1,Fe2,Fe3},  has described in great  detail  the ergodic properties in the case in which $(\Sigma, \sigma)$ is a topologically mixing sub-shift over a finite alphabet. We extend some of his results to this non-compact setting under a combinatorial assumption on the shift and an 
almost-additivity assumption on the sequence. 

Another application we obtain is a formula relating the  pressure for almost-additive sequences  and the Hausdorff dimension of certain geometric constructions (see Section \ref{bow}). We actually  generalise  results obtained by Barreira \cite{b1}.

Finally,  a couple of examples are discussed. We obtain an explicit formula for the pressure of an almost-additive sequence of locally constant functions in the case of the full-shift (Section \ref{ex1}). We also discuss the thermodynamic formalism of an almost-additive sequence of continuous functions that naturally arises in the study of factor maps (see Section \ref{ex2}).

We believe that further interesting applications of the results presented here (or of generalisations of them) can be obtained. For instance, it would be of interest to study the \emph{Jacobi-Perron map}  which is the map $T : (0,1]^{2} \rightarrow (0,1]^{2}$
defined by
\begin{equation*}
T(x,y) = \left( \frac{y}{x}- \left[ \frac{y}{x} \right] , \frac{1}{x} -\left[ \frac{1}{x} \right] \right).
\end{equation*}
This map is closely related to the Jacobi-Perron algorithm of simultaneous approximation of real numbers \cite{sc}. It is a map which is Markov over a countable partition \cite{ma} and it is not conformal. 

It is also plausible that  a fine analysis of Lyapunov exponents in the spirit of \cite{bg} can also be achieved for certain classes of non-conformal maps that are also non-uniformly hyperbolic. This might be done using the techniques developed here.
 
\section{Definition of almost-additive Gurevich pressure} \label{sec:def}

Let $(\Sigma, \sigma)$ be a one-sided Markov shift
over a countable alphabet $S$. This means that there exists a matrix
$(t_{ij})_{S \times S}$ of zeros and ones (with no row and no column
made entirely of zeros) such that
\begin{equation*}
\Sigma=\left\{ x\in S^{\N_0} : t_{x_{i} x_{i+1}}=1 \ \text{for every $i
\in \N_0$}\right\}.
\end{equation*}
The \emph{shift map} $\sigma:\Sigma \to \Sigma$ is defined by $\sigma(x)=x'$, for $x=(x_n)_{n=0}^{\infty}, x'=(x'_n)_{n=0}^{\infty}, x'_n=x_{n+1}$ for all $n\in \N_{0}$.  Sometimes we simply say that $(\Sigma, \sigma)$ 
is a \emph{countable
Markov shift}. For an admissible word $i_0 \dots i_{n-1}$ of length $n$ in $\Sigma$, we define a cylinder set $C_{i_0 \dots i_{n-1}}$ of length $n$ by 
\begin{equation*}
 C_{i_0 \cdots i_{n-1}}= \left\{x \in \Sigma : x_j=i_j \text{ for } 0 \le j \le n-1 \right\}.
 \end{equation*}
We equip $\Sigma$ with the
topology generated by the cylinders sets. We denote by $\M$ the set of $\sigma$-invariant Borel probability measures on $\Sigma$.
We will always assume $(\Sigma, \sigma)$ to be topologically mixing, that is, for every $a,b \in S$ there exists $N_{ab} \in \N$ such that 
for every $n > N_{ab}$ we have
$C_a \cap \sigma^{-n} C_b \neq \emptyset$.
\begin{defi}\label{subadditive}
Let $(\Sigma,\sigma)$ be a one-sided countable state Markov shift. For each $n\in \N$, let $f_{n}: \Sigma \to \R^{+}$ be a continuous function. 
A sequence $\mathcal{F}= \{ \log f_n \}_{n=1}^{\infty}$ on $\Sigma$ is called
 \emph{sub-additive} if for every $n, m\in \N,x\in \Sigma$, we have 
\begin{equation}\label{sc}
0 < f_{n+m}(x) \leq f_n(x) f_{m}(\sigma^n x).
\end{equation}
\end{defi}
Recall that if $\F$ is a sub-additive sequence, then  by  Kingman's sub-additive ergodic theorem \cite{ki},  there exits a measurable function $f$ with the following property:
Let $\mu\in \M$. If $\log f_n: \Sigma \rightarrow \R \cup \{-\infty\}$ for all $n\in \N$ and $(\log f_1)^{+}\in L^{1}(\mu)$, then for  $\mu$-almost every $x \in \Sigma$, 
\begin{equation*}
\lim_{n \to \infty} \frac{1}{n} \log f_n(x) = f(x) \quad \text{ and } \quad
\lim_{n \to \infty} \frac{1}{n} \int \log f_n(x) \ d\mu= \int f(x)\ d \mu.
\end{equation*}

\begin{defi} \label{aaa}
Let $(\Sigma,\sigma)$ be a one-sided countable state Markov shift. For each $n\in \N$, let  $f_{n}: \Sigma \to \R^{+}$ be a continuous function.
A sequence $\mathcal{F}= \{ \log f_n \}_{n=1}^{\infty}$ on $\Sigma$ is called \emph{almost-additive} if there exists a constant $C\geq0$ such that for every $n,m\in \N, x\in \Sigma$, we have  
\begin{equation} \label{A1}
 f_n(x) f_{m}(\sigma^n x) e^{-C} \leq f_{n+m}(x),
 \end{equation}
and
\begin{equation} \label{A2}
 f_{n+m}(x) \leq  f_n(x) f_{m}(\sigma^n x) e^{C}.
    \end{equation}
   \end{defi}
In most parts of this paper, we will assume the sequence $\F$ to be almost-additive. It should be  pointed out that even in the compact setting several results we present here are not valid under weaker assumptions.

\begin{rem} \label{cla}
If $f:\Sigma\rightarrow \R$ is a continuous function, the Birkhoff sums of $f$ form an almost-additive sequence of continuous functions, in this case the constant $C$ of Definition \ref{aaa} is equal to zero. Indeed,  for every $n \in \N, x\in \Sigma$, define $f_n:\Sigma\rightarrow \R^{+}$ 
by $f_n(x)=e^{f(x)+f(\sigma x) +\dots+f (\sigma^{n-1}x)}$. Then the sequence $\{\log f_n\}_{n=1}^{\infty}$ is additive. 
This construction is the link that ties up the thermodynamic formalism for a continuous function with that of sequences of continuous functions.
\end{rem}

One of the important ingredients in thermodynamic formalism is the regularity assumptions on the (sequence of) continuous functions. 
Several  results depend upon this hypothesis. In the rest of paper, we will always assume the following regularity conditions.

\begin{defi}\label{bowen}
Let $(\Sigma, \sigma)$ be a one-sided countable Markov shift. For each $n\in \N$, let $f_{n}: \Sigma \rightarrow \R^{+}$ be continuous. A sequence 
$\mathcal{F}= \{ \log f_n \}_{n=1}^{\infty}$ on $\Sigma$ is called a \emph{Bowen} 
sequence if there exists $M \in \R^{+}$ such that
\begin{equation}\label{bowenbound}
 \sup \{ A_n : n \in \NÊ\} \leq M, 
\end{equation}
where
\[A_n= \sup \left\{ \frac{f_n(x)}{f_n(y)} : x,y  \in \Sigma, x_i=y_i \textrm{ for } 0 \leq i \leq n-1\right\}.\]
\end{defi}

The definition above is related to a regularity assumption introduced by Bowen when studying classic thermodynamic formalism. Indeed, for $f\in C(\Sigma)$, let $V_n(f):=\sup\{|f(x)-f(x')|: x, x'\in \Sigma,  x_i=x'_{i}, 0\leq i\leq n-1\}$ and $\sigma_nf:=\sum_{i=0}^{n-1}f\circ \sigma^{i}$. The \emph{Bowen class} is defined by $\{f\in C(\Sigma): \sup_{n\in \N}V_{n}(\sigma_nf)<\infty\}$  (see \cite{bo, w3}). Most of the thermodynamic formalism  for continuous functions is well developed when the function belongs to the Bowen class. 
It is easy to see that given a function in the Bowen class, the additive sequence constructed  in Remark \ref{cla} is a Bowen sequence. We can think of a Bowen sequence as the natural generalisation of  a function belonging to the  Bowen class. 

We also remark that our definition of Bowen sequence restricted to a compact set is equivalent to the definition  of sequences of bounded variation given  by 
Barreira in \cite{b2} (see also \cite{m} for similar conditions). 
It is plausible that results presented in sections \ref{sec:def} and \ref{sec:var} could be extended to a larger class  of sequences of continuous functions satisfying a tempered 
condition of the type $\lim_{n\rightarrow \infty}A_n/n=0$ (see \cite{b2} for more about this condition).

The aim of this section is to provide a good definition of pressure for almost-additive 
sequences of continuous functions defined over a countable Markov shift.  There exist several definitions of pressure for sub-additive (and hence for almost-additive) sequences defined 
over compact spaces. Indeed,  Falconer \cite{f} gave a definition that behaves well for sub-shifts of finite type defined over finite alphabets.   Cao,  Feng and Huang \cite{cao} 
gave one based on $(n,\epsilon)-$sets and Barreira \cite{b1, b2} studied a definition using the theory of dimension-like characteristics  
developed by Pesin \cite{Pe}. Mummert also gave a definition in the same spirit \cite{m}.   
We stress that our situation completely differs from the above  since our phase space is no longer compact. It is important that our definition does not depend upon 
the metric (as in the case of Cao et al \cite{cao}), so that it satisfies a variational principle (for a discussion of this issue see Section \ref{sec:var}). 
Also note that we cannot continuously  extend $\F$ to any compactification of the space $\Sigma$ (which is what is needed if we want to extend the definitions 
of  Barreira \cite{b1,b2} or Mummert \cite{m}) because the sequence $\F$ is not assumed to be bounded. Finally, note that if we extend the definition given by Falconer \cite{f}   to the countable Markov shift setting, then only a narrow class of shifts would satisfy the variational principle (see \cite{gu1, gu2}).

For countable Markov shifts, the thermodynamic formalism has been developed by Mauldin and Urba\'nski \cite{mu1, mu} for a certain class of Markov shifts with  combinatorics  close to that of the full-shift and in full generality by Sarig \cite{s1,s2,s3,s4}. 
The definition we propose is both a generalisation of the Gurevich pressure for Markov shifts over a countable alphabet introduced by Sarig in \cite{s1} and the  
pressure for almost-additive sequences on compact spaces introduced independently by Barreira in \cite{b2} and by Mummert \cite{m}.

\begin{defi}
Let $(\Sigma,\sigma)$ be a countable state Markov shift. For $a \in S$ and an almost-additive Bowen  sequence $\mathcal{F}= \{ \log f_n \}_{n=1}^{\infty}$ on $\Sigma$, we define the \emph{partition function} by
\begin{equation}
Z_n(\F,a)= \sum_{\sigma^n x= x} f_n(x) \chi_{C_a}(x),
\end{equation}
where $\chi_{C_a}(x)$ is the characteristic function of the cylinder $C_a$.
\end{defi}

\begin{defi}
Let $(\Sigma,\sigma)$ be a countable state Markov shift. The \emph{almost-additive Gurevich pressure} of an almost-additive sequence $\F$ on $\Sigma$ is defined by
\begin{equation}
P(\F)= \lim_{n \to \infty} \frac{1}{n} \log Z_n(\F,a).
\end{equation}
\end{defi}

\begin{rem} If $f:\Sigma \to \R$ is a continuous function of summable variations, define 
$\F=\{\log f_n\}_{n=1}^{\infty}$ where $f_n(x)=e^{f(x)+f(\sigma x) +\dots+f (\sigma^{n-1}x)}$ for every $n \in \N, x\in \Sigma$. Then the almost-additive Gurevich pressure of $\F$ 
coincides with the Gurevich pressure of $f$ defined by Sarig in \cite{s1}. Also, if the Markov shift $(\Sigma,\sigma)$ is defined over a finite alphabet, then the 
almost-additive Gurevich pressure coincides with the almost-additive pressure  introduced by Barreira in \cite{b2}(compare also with Mummert's definition in \cite{m}).  
 \end{rem}  

The rest of this section is devoted to prove that the almost-additive Gurevich pressure is actually well defined.
Throughout the rest of this section, we assume that $S$ is a countable alphabet and $(\Sigma, \sigma)$ is a topologically mixing countable Markov shift. 
For a Bowen sequence $\F=\{ \log f_n \}_{n=1}^{\infty}$ on $\Sigma$, we let $M$ and $A_n, n\in \N$ be defined as in Definition \ref{bowen}.
\begin{lema} \label{1}
Let $\mathcal{F}= \{ \log f_n \}_{n=1}^{\infty}$ be a Bowen sequence on $\Sigma$ that satisfies equation \eqref{A1}. Then there exists a constant $k \in \R$ such that 
for every $a\in S, n,m \in \N$, we have
\begin{equation*}
Z_n(\F,a) Z_m(\F,a)  \leq Z_{n+m}(\F,a) e^{k}.
\end{equation*}
 \end{lema}

\begin{proof}
For $a\in S$, let $x \in \Sigma \cap C_a$ be a periodic point of period $n$, that is $\sigma^nx=x$. 
Then we can write $x=(x_0, x_1, \dots, x_{n-1}, x_0, x_1, \dots)$, where $x_0=a$. Let $A$ denote
the 
admissible word $x_0 \dots x_{n-1}$. Consider now $\tilde{x} \in \Sigma \cap C_a$, a periodic point of period $m$. Again, we can write 
$\tilde{x}=(\tilde{x}_0,\tilde{x}_1,  \dots, \tilde{x}_{m-1}, \tilde{x}_0, \tilde{x}_1 \dots)$, where $\tilde{x}_0=a$. Let $B$ denote
the 
admissible word $\tilde{x}_0 \dots \tilde{x}_{m-1}$. Let us consider now the point $x' \in Ê\Sigma \cap C_a$ obtained by concatenating the admissible words $A$ and $B$, that is
\[x'=ABABABABABAB \dots := (AB)^{\infty}.\]
Clearly $\sigma^{n+m} x'=x'.$
Since $\F$ satisfies equation \eqref{A1}, we obtain
\begin{eqnarray*}
e^{-C} f_n(x') f_{m}(\sigma^n x')  \leq f_{n+m}(x'),    \end{eqnarray*}
 which implies that
 \begin{eqnarray*} e^{-C} f_n(x') f_{m}(\sigma^n x') f_n(x)f_m(\tilde{x}) \leq  f_{n+m}(x')f_n(x)f_m(\tilde{x}).
   \end{eqnarray*}
  Hence
  \begin{eqnarray*} f_n(x)f_m(\tilde{x}) \leq e^{C}  f_{n+m}(x') \frac{f_n(x) f_m(\tilde{x})}{f_n(x') f_m(\sigma^n x')}.
    \end{eqnarray*}
Since for every $0\leq i \leq m-1$ we have $(\sigma^n x')_i= \tilde{x}_i$, we obtain
\begin{equation*}
 f_n(x)f_m(\tilde{x}) \leq e^{C}  f_{n+m}(x') A_n A_m \leq e^C M^2 f_{n+m}(x').\end{equation*}
The result now follows  setting $e^k=  e^C M^2$.
\end{proof}

Let $\mathcal{F}= \{ \log f_n \}_{n=1}^{\infty}$ be a sequence of continuous functions on $\Sigma$  and let $m  \in \N$. For $w:\Sigma \to \R$ a
continuous function, we define
\begin{equation*}
\left(L_{\F}^m w  \right) (x) := \sum_{\sigma^m z =x} f_m(z) w(z)\quad  \text{ for }x\in \Sigma.
\end{equation*}
Let $Y \subset \Sigma$ be a topologically mixing finite state Markov shift. Then for a sequence $\F\vert_Y:=\{\log f_n\vert_Y\}_{n=1}^{\infty}$, we have 
\begin{equation*}
\left(L_{\F |_Y}^m w  \right) (y) = \sum_{\sigma^m z =y} f_m(z) w(z) \chi_Y(z) \quad \text{ for } y\in Y.
\end{equation*}

\begin{lema} \label{bounds}
Let $\F=\{ \log f_n \}_{n=1}^{\infty}$ be a Bowen sequence on $\Sigma$.
Let $Y \subset \Sigma$ be a topologically mixing finite state Markov shift such that $Y \cap C_a \neq \emptyset, a\in S$. Define
\[Z_n(Y,\F,a)= \sum_{\sigma^n x= x} f_n(x) \chi_{C_a \cap Y}(x).\]
Then for each $a\in S, x \in C_a \cap Y, n \in  \N$,
\begin{equation*}
 \frac{1}{M}L_{\F |_Y}^n \left( \chi_{C_a} \right) (x)\leq Z_{n}(Y, \F, a)\leq M L_{\F |_Y}^n \left( \chi_{C_a} \right) (x).
\end{equation*}
\end{lema}

\begin{proof}
For $a\in S$, let $x \in C_a  \cap  Y$ and write $x=a \tilde{x}$. Then
\begin{eqnarray*}
 L_{\F |_Y}^n \left( \chi_{C_a} \right) (a\tilde{x})=
 \sum_{\sigma^n z = a \tilde{x}} f_n |_{Y} (z) \chi_{C_a \cap Y}(z) =
  \sum_{z\in Y, z=a z_1 z_2 \dots z_{n-1} a \tilde{x}}      f_n |_{Y} (z).
  \end{eqnarray*}
Let $\overline{x} \in C_a \cap Y$ be such that $\sigma^n \overline{x}=\overline{x}$.  Then we can 
write the point  $\overline{x} = (a, \overline{x}_1, \dots, \overline{x}_{n-1}, a, \overline{x}_1, \dots, \overline{x}_{n-1}, \dots)$. 
Now set $az_1z_2 \dots z_{n-1} = a\overline{x}_1 \dots \overline{x}_{n-1}$. In this way, we have $a\overline{x}_1 \dots \overline{x}_{n-1}a \tilde{x} \in C_a \cap Y$.
Since $\F$ is in the Bowen class,
\begin{equation*}
\frac{f_n(\overline{x})}{f_n|_{Y} (a\overline{x}_1 \dots \overline{x}_{n-1}a \tilde{x})}\leq M.
 \end{equation*}
  Therefore,
  \[ Z_n(Y, \F,a) \leq
M L_{\F |_Y}^n \left( \chi_{C_a} \right) (x).\]
In order to prove the other inequality, for $ay_1 \dots y_{n-1}a\tilde{x} \in Y$, we define the point  $x'=(a, y_1, \dots, y_{n-1}, a, y_1 \dots, y_{n-1},a, \dots)$ of period $n$. 
In this way, $x' \in C_a \cap Y$ and $\sigma^n x' = x.$ Therefore,
\begin{equation*}
\frac{f_n |_{Y} (ay_1 \dots y_{n-1} a \tilde{x})}{f_n(x')}\leq M.
\end{equation*}
Hence
\[\frac{1}{M} L_{\F |_Y}^n \left( \chi_{C_a} \right) (x) \leq Z_n(Y, \F,a).\]
\end{proof}

\begin{lema}  \label{3}
Let  $\mathcal{F}= \{ \log f_n \}_{n=1}^{\infty}$ be a Bowen sequence on $\Sigma$ that satisfies equation \eqref{A1} 
and let $Y  \subset \Sigma$ be a topologically mixing finite state Markov shift. Then there exists $\beta >0$ such that for each $a\in S, x\in C_{a}\cap Y, n\in\N$, 
\begin{equation*}
\left( L_{\F|_Y}^n \chi_{C_a}\right)(x) \geq \beta^n e^{-(n-1) C}.
\end{equation*}
Moreover,
\[Z_n(Y, \F,a) \geq \frac{1}{M}   \beta^n e^{-(n-1) C},\]
and so
\[Z_n( \F,a) \geq \frac{1}{M}   \beta^n e^{-(n-1) C},\]
where $C$ is defined as in equation (\ref{A1}).
\end{lema}

\begin{proof}
For $x \in Y\cap C_a$, we have that
\begin{eqnarray*}
\left( L_{\F|_Y}^n \chi_{C_a}\right)(x) = \sum_{\sigma^n z=x} f_n|_{Y}(z) \chi_{C_a \cap Y}(z)=
\sum_{z\in Y, z=az_1 \dots z_{n-1} x}  f_n|_{Y}(z).   \end{eqnarray*}
If $x= a \tilde{x}$, we have
\begin{eqnarray*}
\left( L_{\F|_Y}^n \chi_{C_a}\right)(a\tilde{x}) = \sum_{z\in Y, z=az_1 z_2 \dots z_{n-1}a \tilde{x}} f_n|_{Y}(z) \geq \min_{z \in Y} f_n|_{Y}(z).
\end{eqnarray*}
In virtue of equation \eqref{A1} we see that for every $x \in \Sigma$, 
\begin{align*}
f_n(x) &\geq e^{-C}f_{n-1}(\sigma x) f_1(x) \geq e^{-2C}f_{n-2}(\sigma^2 x) f_1(\sigma x)f_1(x)\\
&\geq \dots \\
&\geq e^{-(n-1)C}f_1(\sigma^{n-1}x)f_{1}(\sigma^{n-2}x) \dots f_1(\sigma x)f_1(x).
\end{align*}
Since the space $Y$ is compact and invariant, we set $\beta= \min_{z \in Y}  f_1\vert_Y (z)$. Then, for every $y \in Y$ we have
\[ f_n(y) \geq e^{-(n-1)C}\beta^n.\]
From the above equation, we can conclude that
\[\left( L_{\F|_Y}^n \chi_{C_a}\right)(x) \geq \beta^n e^{-(n-1) C}.\]
The above result together with Lemma \ref{bounds} implies that
\[Z_n(Y, \F,a) \geq \frac{1}{M}   \beta^n e^{-(n-1) C},\]
and so
\[Z_n( \F,a) \geq \frac{1}{M}   \beta^n e^{-(n-1) C}.\]
\end{proof}

 Given $f: \Sigma \to \R$ a continuous function, the \emph{transfer operator} $L_{f}$   applied to  function $g: \Sigma \rightarrow \R$ 
is formally defined  by
\begin{equation} \label{transfer}
\left( L_{f} g \right) (x) := \sum_{\sigma z=x} f(z) g(z) \quad \text{ for every } x\in \Sigma.
\end{equation}
In the following lemma, we will make use of the supremum norm of  a continuous function. For a function $g:\Sigma\rightarrow \R$, define  $\|g\|_{\infty}=\sup\{|g(x)|:x\in \Sigma\}$. 
\begin{lema} \label{upper}
Let  $\mathcal{F}= \{ \log f_n \}_{n=1}^{\infty}$ be a Bowen sequence on $\Sigma$ that satisfies equation \eqref{A2}. 
Then  there exists $\tilde{M} >0$ such that for any $a\in S, n\in\N$
\begin{equation*}
Z_n(\F,a) \leq \tilde{M} e^{C(n-1)} \| L_{f_1} 1 \|_{\infty}^{n},
\end{equation*}
where $C$ is defined as in equation  \eqref{A2}.
\end{lema}

\begin{proof}
Let $a\in S$. In order to prove this lemma, it suffices to show that there exists $\tilde{M} >0$ such that for every $x\in \Sigma$,
\begin{equation*}
Z_n(\F,a) \leq  \tilde{M} \left(L_{\F}^n 1  \right)(x) \leq   \tilde{M} \left(L_{f_1}^n 1  \right)(x) e^{C(n-1)} \leq   
\tilde{M}e^{C(n-1)}\| L_{f_1} 1 \|_{\infty}^{n}.
\end{equation*}
We can show the first inequality using a similar argument as the one used in the proof of Lemma \ref{bounds} (replace $Y$ by $\Sigma$).

Now we show  the second inequality. Note that if $x=a \tilde{x} \in C_a$, we have that
\[\left(L_{\F}^n 1  \right)(x) = \sum_{\sigma^n z = a \tilde{x}}   f_n (z)= \sum_{z\in \Sigma, z=z_0 z_1 \dots z_{n-1} a \tilde{x}} f_n(z).\]
From  equation \eqref{A2}, we have
\begin{align}\label{it}
f_n(z) &\leq e^C f_1(z) f_{n-1}(\sigma z)  \leq e^{2C} f_1(z) f_1(\sigma z) f_{n-2}(\sigma^2 z) &\\
&\leq  e^{(n-1)C} f_1(z) f_1(\sigma z) \dots  f_1(\sigma^{n-1} z).
 \end{align}
On the other hand, the iterations of the transfer operator (see equation \eqref{transfer}) 
\begin{equation}\label{nice}
\left( L_{f_1}^n 1 \right)(x) = \sum_{i_n \dots i_1 x\in \Sigma} f_1(i_1 x)  f_1(i_2 i_1 x) \dots  f_1(i_n \dots i_1 x).
\end{equation}
Therefore,
\begin{align*}
\left(L_{\F}^n 1  \right)(a\tilde{x}) &= \sum_{\sigma^{n}z= a \tilde{x}} f_n(z)
\leq \sum_{\sigma^{n}z= a \tilde{x}} f_1(z) \dots f_1(\sigma^{n-1}z) e^{(n-1)C} \text{ (by (\ref{it}))} \\
& = \sum_{z\in \Sigma, z=z_0 z_1 \dots z_{n-1} a \tilde{x}} f_1(z_{n-1} a\tilde{x})  f_1(z_{n-2}z_{n-1} a\tilde{x})\dots f_1(z_0 \dots z_{n-1} a \tilde{x}) e^{(n-1)C} \\
&= e^{(n-1)C}L_{f_1}^n 1 (a \tilde{x}) \text{ (by (\ref{nice})) },
\end{align*}
which implies the second inequality.
Finally, a direct computation shows that 
\[\left( L_{f_1}^n 1  \right)(x) \leq  \|  L_{f_1} 1  \|_{\infty}^n \textnormal{ for every } x\in C_{a}.\] 
\end{proof}

\begin{lema}
Let  $\mathcal{F}= \{ \log f_n \}_{n=1}^{\infty}$ be an almost-additive Bowen sequence on $\Sigma$. Then for $a\in S$,
\begin{enumerate}
\item The limit
\[\lim_{n \to \infty} \frac{1}{n} \log Z_n(\F,a) \]
exists and it is not minus infinity.
\item If $ \|  L_{f_1} 1  \|_{\infty} < \infty$, then
\[\lim_{n \to \infty} \frac{1}{n} \log Z_n(\F,a) \neq \infty.\]
\end{enumerate}
\end{lema}

\begin{proof}
The fact that the limit exists follows from Lemma \ref{1}. 
It follows from Lemmas \ref{1} and \ref{3} 
that the limit is not minus infinity. 
The second part of the lemma is a consequence of Lemma \ref{upper}. 
  \end{proof}

\begin{prop}\label{indp}
Let  $\mathcal{F}= \{ \log f_n \}_{n=1}^{\infty}$ be  an almost-additive Bowen sequence on $\Sigma$. Then
the limit
\[\lim_{n \to \infty} \frac{1}{n} \log Z_n(\F,a) \]
is independent of the symbol $a\in S$.
\end{prop}

\begin{proof}
It is independent of the symbol $a \in S$ because $(\Sigma, \sigma)$ is topologically mixing. Indeed, we will prove that given any two symbols $a,b \in S$ there exist constants
$C'>0$  and $k(a,b) \in \N$ such that
\[Z_{n}(\F, a) \leq C' Z_{n+2k(a,b)}(\F, b),\]
from which the result follows.

Let $x=(a, x_1, x_2, \dots, x_{n-1}, a ,\dots) \in C_a$ be a periodic point of period $n$, that is  $\sigma^n x=x$. Since $\Sigma$ is topologically mixing, 
there exists $N_{ba} \in \N$ such that for any $m \geq N_{ba}$ there exists an admissible
word of length $(m-1)$ given by $by_1 \dots y_{m-2}$ such that $by_1 \dots y_{m-2}a$ is
an admissible  word of length $m$. Similarly, there exists $N_{ab} \in \N$ such that for any $m \geq N_{ab}$ there exists an admissible  word of length $(m-1)$ given 
by $az_1 \dots z_{m-2}$ such that $az_1 \dots z_{m-2}b$ is an admissible  word of length $m$. Set
$k= \max \{ N_{ba}, N_{ab} \}$ and $m=k+1$. Now consider the periodic point $\tilde{x} \in \Sigma$ satisfying $\sigma^{n+2k} \tilde{x} =\tilde{x}$, where
\[\tilde{x} = (b,y_1, \dots, y_{m-2},a,x_1, x_2, \dots, x_{n-1}, a,z_{1}, \dots, z_{m-2}, b, \dots).\]
Note that
\begin{equation}\label{it1}
f_n(x) = \frac{f_n(x)}{f_{n+2k} (\tilde{x})}f_{n+2k} (\tilde{x}).
\end{equation}
Since $\F$ satisfies equation \eqref{A1}, we have that
\begin{align*}
\frac{f_n(x)}{f_{n}(\sigma^k \tilde{x})} &\geq \frac{f_n(x) f_k(\tilde{x})}{f_{n+k}(\tilde{x}) e^C} \geq
\frac{f_n(x) f_k(\tilde{x}) f_k(\sigma^{n+k} \tilde{x})}{e^{2C} f_{n+2k}(\tilde{x})}\\
&\geq \left( \frac{f_n(x)}{e^{2C} f_{n+2k}(\tilde{x})}  \right)
\frac{1}{M^2} \sup \left\{ f_k(u) : u \in C_{by_1 \dots y_{k-1}} \right\}
\sup \left\{ f_k(u) : u \in C_{az_1 \dots z_{k-1}} \right\}.
 \end{align*}
Recall that $f_k>0$. Therefore, there exist positive constants $C_1, C_2$ such that
\[C_1 \leq \sup \left\{ f_k(u) : u \in C_{by_1 \dots y_{k-1}} \right\} \textrm{ and }
C_2 \leq  \sup \left\{ f_k(u) : u \in C_{az_1 \dots z_{k-1}} \right\}.\]
Hence
\begin{eqnarray*}
\frac{f_n(x)}{f_{n+2k}(\tilde{x})} \leq \frac{e^{2C}M^2 }{C_1 C_2} \frac{f_n(x)}{f_{n}(\sigma^k \tilde{x})}.
\end{eqnarray*}
Therefore, using inequality (\ref{it1}) 
\begin{equation*}
f_n(x) \leq \frac{e^{2C} M^3 }{C_1 C_2}f_{n+2k}(\tilde{x}).
 \end{equation*}
Thus there exists a constant $C' >0$ such that
\[ \sum_{\sigma^n x=x} f_n(x) \chi_{C_a}(x) \leq C'  \sum_{\sigma^{n+2k} x=x} f_{n+2k}(x) \chi_{C_b}(x).\]\end{proof}

\begin{rem}
Let us stress that Lemma \ref{upper}  holds under the assumption that the Bowen sequence $\F$ satisfies equation (\ref{A2}). 
On the other hand, Lemmas \ref{1} and \ref{3} hold when the Bowen sequence $\F$ satisfies equation (\ref{A1}).
\end{rem}

\section{The variational principle} \label{sec:var}
One of the major results in the classical thermodynamic formalism is that the topological 
pressure satisfies the variational principle \cite{w1}. It states that  if we consider a dynamical system  defined on a compact metric  space and a continuous function $\phi$ the following equality holds,
\begin{equation} \label{ip}
P(\phi)= \supÊ\left\{ h(\mu) + \int \phi \ d\mu : \mu \in \M  \right\}.
\end{equation}
In this context, the topological pressure is defined by means of 
$(n, \epsilon)-$separated sets (see \cite[Chapter 9]{w2}). This notion depends upon the metric.
Since in the compact setting all the metrics generating the same topology are equivalent,
the value of the pressure does not  depend upon the metric. Recently, following the same approach, Cao, Feng and Huang \cite{cao} defined the pressure and proved the variational principle for sub-additive sequences of continuous functions defined on a compact metric space.
Under different assumptions, Barreira \cite{b2}, Falconer \cite{f} and Mummert \cite{m} have also obtained variational principles.

In the non-compact setting, the definition of topological pressure obtained using $(n, \epsilon)-$separated sets has 
several problems. Most notably, different metrics generating the same topology can yield different values for the 
topological pressure. Let us stress that the right hand side of the equality \eqref{ip} only depends on the 
Borel structure of the space and not on the metric. Therefore, a notion of pressure 
satisfying the variational principle need not depend upon the metric of the space. 

The definition we proposed in the previous section does not depend on the metric. We stress again that it is both a 
generalisation of the notion introduced by Sarig \cite{s1} and 
of a formula for the pressure of almost-additive sequences  obtained by Barreira \cite{b2} and Mummert \cite{m} (both of which satisfy a 
version of equality (\ref{ip})). The main result of this section is that the almost-additive Gurevich pressure 
satisfies the variational principle. We also prove that it is well approximated by its restriction to compact  
invariant sets.

Recall that, when  $(\Sigma, \sigma)$ is a topologically mixing finite state Markov shift and $\F$ is an almost-additive sequence on $\Sigma$ satisfying the tempered condition 
$\lim_{n\rightarrow \infty}(\log A_n)/n=0$ (where $A_n$ is defined in Definition \ref{bowen}), Barreira \cite[Theorem 2]{b2} proved that 
the pressure he defined 
satisfies the following formula:
\[P(\F) = \lim_{n \to \infty} \frac{1}{n} \log  \left(\sum_{\sigma^n x= x} f_n(x)\right).  \]
Thus,
\[P(\F) = \lim_{n \to \infty} \frac{1}{n} \log  \left(\sum_{\sigma^n x= x} f_n(x)\right) =  \lim_{n \to \infty} \frac{1}{n} \log  \left(\sum_{\sigma^n x= x} f_n(x)\chi_{C_a}(x)\right). \]
We prove now  that the almost-additive pressure defined over a countable Markov shift can be approximated by the pressure on finite Markov shifts.

\begin{prop}\label{approx} 
Let $(\Sigma, \sigma)$ be a topologically mixing countable state Markov shift and  $\mathcal{F}= \{ \log f_n \}_{n=1}^{\infty}$ be an almost-additive 
Bowen sequence on $\Sigma$. Then
\begin{equation*}
P(\F)= \sup \{ P(\F\vert_Y) : Y \subset \Sigma \textrm{ a topologically mixing finite state Markov shift} \}.
\end{equation*}
\end{prop}

\begin{proof} 
Let $Y \subset \Sigma$ be a topologically mixing finite state Markov shift. Then clearly
\begin{align*}
 P(\F |_Y) &= \lim_{n \to \infty} \frac{1}{n} \log \left(\sum_{\sigma^n x= x} f_n(x) \chi_{(Y \cap C_a)}(x) \right) \\
 &\leq \lim_{n \to \infty} \frac{1}{n} \log \left(\sum_{\sigma^n x= x} f_n(x)\chi_{C_a}(x)\right) = P(\F).
 \end{align*}
Therefore
\begin{equation*}
P(\F) \geq \sup \{ P(\F\vert_Y) : Y \subset \Sigma \textrm{ a topologically mixing finite state Markov shift} \}.
\end{equation*}
In order to prove the reverse inequality, we assume that $P(\F) <\infty$.  The other case can be proved in a similar way.
The following proof is  close to that of Theorem 2 \cite{s1} and we identify a countable alphabet $S$ with $\N$. Since $\F$ is an almost-additive Bowen sequence, let us assume 
that $\F$ satisfies equations (\ref{A1}), (\ref{A2}) and (\ref{bowenbound}). 
We have  
\begin{equation*}
P(\F)=\lim_{n\rightarrow \infty}\frac{1}{n}\log Z_n(\F,a)=\sup_{n}\frac{1}{n}\log Z_n(\F,a). 
\end{equation*}
Hence given $\epsilon >0$, there exists $m>{C}/{\epsilon}$ (where $C$ is defined as in Definition \ref{aaa}) $m \in \N$ such that
\begin{equation*}
P(\F) < \frac{1}{m} \log Z_m(\F, a) + \epsilon.
\end{equation*}
On the other hand, note that every periodic orbit belongs to a set of the form
$\Sigma \cap \Sigma_M$, where $\Sigma_M$ is the full-shift on $M$ symbols. Therefore,
\[Z_m(\F, a) = \lim_{M \to \infty} Z_m(\Sigma_M \cap \Sigma , \F, a).\]
Clearly the sequence in the limit is increasing in $M$. Hence, 
given 
$\epsilon >0$ there exists $M \in \N$ such that
\[\frac{1}{m} \log Z_m(\F, a)< \frac{1}{m} \log Z_m(\Sigma_M \cap \Sigma,\F, a) + \epsilon.\]
Adding a finite number of states to $\{1, 2, \dots , M\}$, it is possible to construct a topologically mixing (finite state) Markov shift $Y \subset \Sigma$ such that
\[\frac{1}{m} \log Z_m(\F, a)< \frac{1}{m} \log Z_m(Y,\F, a) + \epsilon.\]
Indeed, this follows since for every symbol $i,j$ belonging to 
the set $\{1,2,\cdots, M\}$, 
there exists $\{b_1^{ij}, b_2^{ij}, \dots , b_{n(i,j)}^{ij}\} \subset \N$ such that the 
word  
$i b_1^{ij}b_2^{ij} \dots  b_{n(i,j)}^{ij} j$ is an admissible word in $\Sigma$. 
Adding all the symbols $b_m^{ij} \in \N$ obtained this way and taking the closure on $\Sigma$, 
we obtain the required topologically mixing finite state Markov shift $Y$.

Now set $a_n=\log Z_n(Y, F, a)$. Then, by equation (\ref{A1}), $a_n+a_m\leq a_{n+m}+C$. Letting $n=km+r$, for $r=0,1,\dots, k-1$, we obtain
\begin{equation*}
\frac{ka_m+a_r}{km+r}\leq \frac{a_{km+r}+kC}{km+r}\leq \frac{a_n}{n}+\epsilon. 
\end{equation*}
Letting $n\rightarrow \infty$, we obtain
\begin{equation*}
\frac{1}{m} \log Z_m(Y,\F, a) \leq P(F\vert_Y) + \epsilon.
\end{equation*}
Therefore, we have
\[P(\F) \leq P(Y, \F) + 3 \epsilon.\]
Hence
\begin{equation*}
P(\F) \leq \sup \{ P(\F\vert_Y) : Y \subset \Sigma \text{ a topologically mixing finite state Markov shift }\}.
\end{equation*}
\end{proof}

\begin{rem}
We stress that in the proof we only used equation (\ref{A1}).
\end{rem}

The following result is a consequence of Proposition \ref{approx}.
\begin{coro}
Let $(\Sigma, \sigma)$ be a topologically mixing countable state Markov shift and $\mathcal{F}= \{ \log f_n \}_{n=1}^{\infty}$ be an almost-additive Bowen sequence on 
$\Sigma$. Then
\begin{equation*}
P(\F)= \sup \{ P(\F\vert_K) : K \subset \Sigma \textrm{ compact and } \sigma^{-1}(K)=K \}.
\end{equation*}
\end{coro}

\begin{coro}
Let $(\Sigma, \sigma)$ be a topologically mixing countable state Markov shift and $\F$ be an almost-additive Bowen sequence on $\Sigma$ with $\vert \vert L_{f_1}1 \vert \vert _{\infty}<\infty$. Then the pressure function $t \mapsto P(t\F)$ is convex.\end{coro}

\begin{proof}
The pressure function  defined on a compact invariant set is convex. Therefore, the result follows because the supremum of convex functions is a convex function.
\end{proof}

Our definition of almost-additive pressure satisfies the variational principle.

\begin{teo}\label{main1}
Let $(\Sigma, \sigma)$ be a topologically mixing countable state Markov shift and $\F$ be an almost-additive Bowen sequence on $\Sigma$, with $\sup f_1 < \infty$.
Then
\begin{align*}
P(\F)&=\sup \left\{ h(\mu) + \lim_{n\to \infty} \frac{1}{n} \int \log f_n \  d\mu : \mu \in \mathcal{M} \textrm{ and }  
\lim_{n \rightarrow \infty}\frac{1}{n} \int \log f_n \  d\mu \neq -\infty
\right\}\\
&=\sup \left\{ h(\mu) + \int \lim_{n\to \infty}\frac{1}{n}\log f_n \  d\mu : \mu \in \mathcal{M} \textrm{ and }  
\int \lim_{n \rightarrow \infty}\frac{1}{n}\log f_n \  d\mu \neq -\infty
\right\}.
\end{align*}
\end{teo}

\begin{proof}
If the pressure is infinite, $P(\F)= \infty$, the variational principle  holds. Indeed,  in virtue of Proposition \ref{approx}, 
there exists a sequence of topologically mixing finite state Markov shifts $\{Y_n\}_{n=1}^{\infty}$, with $Y_n \subset \Sigma$, for every $n \in \N$, such that\begin{eqnarray*}
\infty= P(\F) =  \lim_{n \to \infty}  \sup \left\{ h(\mu) + \lim_{n\to \infty} \frac{1}{n} \int \log f_n \  d\mu : \mu \in \mathcal{M}_{Y_{n}} \right\} &\\
\leq \sup \left\{ h(\mu) + \lim_{n\to \infty} \frac{1}{n} \int \log f_n \  d\mu : \mu \in \mathcal{M} \textrm{ and }  \lim_{n\rightarrow \infty}\frac{1}{n} \int \log f_n \  d\mu \neq -\infty
\right\}.
\end{eqnarray*}
In the rest of the proof we will assume $P(\F)<\infty$.
Since $\F$ is almost-additive, we assume that it satisfies equations (\ref{A1}) and (\ref{A2}). We first show the second equality in Theorem \ref{main1} by proving for any $\mu\in \M$,
\begin{equation*}
\lim_{n\to \infty} \frac{1}{n} \int \log f_n \  d\mu= \int \lim_{n \rightarrow \infty}\frac{1}{n}\log f_n \  d\mu.
\end{equation*}
To see this, set $g_{n}(x)=f_{n}(x)e^{C}$. Then $\{\log g_n\}_{n=1}^{\infty}$ satisfies the subadditivity condition (equation (\ref{sc}) in Definition \ref{subadditive}). Since
 $\sup f_{1}<\infty$, we have 
$\log g_n: \Sigma\rightarrow \R\cup\{-\infty\}$ for all $n\in \N$ and $(\log g_{1})^{+}\in L^{1}(\mu)$. Therefore, Kingman's subadditive ergodic theorem implies the result.
 
Now we will show the first equality. For a compact subset $Y\subset \Sigma$, denote by $\M_Y$ the set of $\sigma$-invariant Borel probability measures on $Y$.
Note that Barreira \cite{b2}  proved the variational principle for the case when $\Sigma$ is a finite state Markov shift. Therefore, in virtue of Proposition \ref{approx}, 
there exists a sequence of topologically mixing finite state Markov shifts $\{Y_n\}_{n=1}^{\infty}$, with $Y_n \subset \Sigma$, for every $n \in \N$, such that
\begin{eqnarray*}
P(\F) = \lim_{n \to \infty} P(\F\vert _{Y_n}) = \lim_{n \to \infty}  \sup \left\{ h(\mu) + \lim_{n\to \infty} \frac{1}{n} \int \log f_n \  d\mu : \mu \in \mathcal{M}_{Y_{n}} \right\} &\\
\leq \sup \left\{ h(\mu) + \lim_{n\to \infty} \frac{1}{n} \int \log f_n \  d\mu : \mu \in \mathcal{M} \textrm{ and }  \lim_{n\rightarrow \infty}\frac{1}{n} \int \log f_n \  d\mu \neq -\infty
\right\}.
\end{eqnarray*}

In order to prove the other inequality, we adapt the proof of  \cite[Theorem 3]{s1}. We need a version of  \cite[Lemma 4]{s1} for sequences of functions. In the following proof, we identify a countable alphabet $S$
with $\N$. For $m\in\N$, set $C_{\geq m}=\{x\in \Sigma: x_{0}\geq m\}$ and  let $\alpha_m=\{C_1, \dots, C_{m-1}, C_{\geq m}\}$ (see Section \ref{sec:def} for the notation of cylinder sets).  
Let $\mu \in \mathcal{M}$. Then
\begin{equation*}
\lim_{m\rightarrow \infty} \left(h_{\mu}(\sigma, \alpha_m)+\lim_{n\rightarrow \infty}\frac{1}{n}\int \log f_n \ d\mu \right)=h(\mu)+\lim_{n\rightarrow \infty}\frac{1}{n}\int \log f_n \ d\mu.
\end{equation*}
 Fix $m \in \N$ and set $\beta=\alpha_m$. Let $\beta_{0}^{n}=\bigvee_{i=0}^{n}\sigma^{-i}(\beta)$. For each $a_i\in \beta$, let $E_{a_0\dots a_n}:=\cap_{k=0}^{n}\sigma^{-k}(a_k)$.  
For $E\in \beta _{0}^n$ define $f_n[E]=\sup\{f_n(x):x\in E\}$. We have that
\begin{align*}
&\frac{1}{n} \left(H_{\mu}(\beta_{0}^n)+\int \log f_n d\mu \right)\\ 
&\leq  \frac{1}{n}\sum_{a,b\in \beta}\mu(a\cap \sigma^{-n}b) \left(\sum_{\substack{E\subseteq a\cap \sigma^{-n}b, \\ E\in \beta_{0}^n}}\mu(E\vert a\cap \sigma^{-n}b)\log 
\frac{f_n[E]}{\mu(E)} \right)\\
&\leq  \frac{1}{n}\left(\sum_{a,b\in \beta}\mu(a\cap \sigma^{-n}b)\log (\sum_{\substack{E\in a\cap \sigma^{-n}b\\ E\in \beta_{0}^n}}f_n[E])\right)+\frac{1}{n}H_{\mu} 
\left(\beta \vee \sigma^{-n}\beta \right),
\end{align*}
where the last inequality follows from  \cite[Lemma 9.9]{w1}. For $a, b\in \beta$, 
set
\begin{equation*} P_n(a, b)=\frac{1}{n}\log  \left(\sum_{E \subseteq a\cap \sigma^{-n}b, E\in \beta _{0}^n}f_{n}[E]\right).
\end{equation*}
Thus we obtain
\begin{equation}\label{entropy}
\frac{1}{n}H_{\mu}(\beta_{0}^n)+\frac{1}{n}\int \log f_n d\mu \leq \left(\sum_{a,b\in \beta}\mu(a\cap \sigma^{-n}b)P_n(a, b)\right)+\frac{2H_{\mu}(\beta)}{n}.
\end{equation}
In what follows, we will obtain an upper bound for $\limsup_{n\rightarrow \infty}P_{n}(a, b)$. This bound will be different depending on whether both $a$ and $b$ belong to the set 
$\{C_1, \dots, C_{m-1}\}$ or not.
Let $M>0$ be such that $\sup\{f_n(y)/f_n(x):x_{i}=y_{i},  0\leq i\leq n-1\}\leq M$ for all $n\in \N$.

\begin{lema}\label{keylema}
Under the assumptions of Theorem \ref{main1}, we have
\begin{enumerate}
\item If $a, b\neq C_{\geq m}$, then $\limsup_{n\rightarrow \infty}P_{n}(a,b)\leq P(\F)$.\label{partition1}
\item If $a=C_{\geq m}$ or $b=C_{\geq m}$, then there exists $C' \in \R$ such that 
$$\limsup_{n\rightarrow \infty}P_n(a, b)\leq C'.$$\label{partition2}
\end{enumerate}
\end{lema}

\begin{proof}
Our arguments are similar to those in \cite{s1} and make use of the ideas in Proposition \ref{indp}. We first show (\ref{partition1}). Let $\alpha=\{C_{i}:i\in \N\}$ and 
$\alpha_{0}^{n}=\bigvee_{i=0}^{n}\sigma^{-i}(\alpha)$.
Since $\alpha_{0}^n$ is finer than $\beta_{0}^{n}$ (see  \cite[Footnote 3 in Chapter 4]{s4}), we have
$$P_{n}(a,b)\leq \frac{1}{n}\log \left(\sum _{E\subseteq a\cap \sigma^{-n}b, E\in \alpha_{0}^n}f_n[E]\right).$$
We claim that there exist constants $A, k>0$ such that
\begin{equation}\label{keycase1}
\sum _{E\subseteq a\cap \sigma^{-n}b, E\in \alpha_{0}^n}f_n[E]\leq A\sum _{\sigma^{n+2k}x=x, x\in a }f_{n+2k}(x).
 \end{equation}
 Let $E_{ad_1\dots d_{n-1}b}\subset a\cap \sigma^{-n}b$. For  convenience, let $a=C_{N_1}$ and $b=C_{N_2}$, where $1\leq N_1, N_2\leq m-1$. 
Take a point $x=(x_{0,}\dots, x_{n},\dots) \in E_{ad_1\dots d_{n-1}b}$ such that
$f_{n}[E_{ad_1\dots d_{n-1}b}]\leq 2f_{n}(x)$. Then $x\in C_{N_1 {\bar{d}}_1 {\bar{d}}_2 \cdots {\bar{d}}_{n-1} N_2}\subset E_{ad_1\dots d_{n-1}b}$, for some ${\bar{d}}_i\in \N, 
C_{\bar{d}_i} \subseteq C_{d_i}, 1\leq i\leq n-1$.  Using the same arguments used to prove Proposition \ref{indp},
we construct $\tilde{x}\in C_{N_1}$ such that $\sigma^{n+2k}\tilde{x}=\tilde{x},$ in  the  following way. Since $N_1 \bar{d}_1 \dots \bar{d}_{n-1}N_2$ is an admissible word of 
length $(n+1)$ in $\Sigma$, using the same notation as in the proof of  Proposition \ref{indp}, set $k=\max\{N_{N_1N_1}, N_{N_2N_1}\}$.
Define $y_{1}\dots y_{k-1}$ and $z_{1}\dots z_{k-1}$ so that $N_1y_1\dots y_{k-1}N_1$ and $N_2z_1\dots z_{k-1}N_1$ are allowable words in $\Sigma$. Set $A=N_1y_{1}\dots y_{k-1}N_1\bar{d}_{1}\dots \bar{d}_{n-1}N_2z_{1}\dots z_{k-1}$ and define $\tilde x=A^{\infty}$. Clearly,
\begin{equation*}
f_n[E_{ad_1\dots d_{n-1}b}]\leq \frac{2f_{n}(x)}{f_{n+2k}(\tilde x)}f_{n+2k}(\tilde x).
\end{equation*}
Note that $x_i=({\sigma}^{k}\tilde x)_{i}$ for $0\leq i\leq n-1$. Therefore, we obtain
\begin{equation*}
\frac{f_{n}(x)}{f_{n}(\sigma^{k}\tilde x)}\leq M.
\end{equation*}
Approximating ${f_{n}(x)}/{f_{n+2k}(\sigma^{k}\tilde x)}$ by an argument similar to that in the proof of Proposition \ref{indp}, we obtain (\ref{keycase1}). Therefore, we have
\begin{align*}
\limsup_{n\rightarrow \infty}P_{n}(a,b)&\leq \limsup_{n\rightarrow \infty}\frac{1}{n}\log \left(\sum _{\sigma^{n+2k}x=x, x\in C_{N_1}}f_{n+2k}(x) \right)\\
& = \lim_{n\rightarrow\infty}\frac{1}{n}\log Z(\F, N_1)=P(\F).
\end{align*}
Next we show Lemma \ref{keylema} (\ref{partition2}). We consider the case when $a=b=C_{\geq m}$. Other cases can be shown similarly.  Let $n\geq 3$ be fixed. Write $\{E \in \beta_{0}^{n}: E\subset a\cap \sigma^{-n}b \}=\cup_{i, j, k}A_{i, j, k}$, where 
\[ A_{i,j,k}=\left\{E_{a^{i}d_{1}\dots d_{j}b^{k}}\in \beta_{0}^n: a=b=C_{\geq m}, d_1, d_j\neq C_{\geq m}\right\},\]
 $i+j+k=n+1$. We first consider the case when $j\geq 2$. For $j= 0, 1$, we make a similar argument. Define for each $i, j, k$,
\begin{equation*}
S_{i,j,k}=\sum_{E \in A_{i,j,k}}f_n[E].
\end{equation*}
We first find an upper bound for $S_{i,j,k}$. Fix $i,j$ and $k$. Let $E_{a^{i}d_{1}\dots d_{j}b^{k}}\in A_{i, j, k}$. For a point $x\in E_{a^{i}d_{1}\dots d_{j}b^{k}}$, we have 
by equation (\ref{A2}) 
\begin{equation*}
f_n(x)\leq  f_{1}(x)f_{n-1}(\sigma x)e^C \leq f_{1}(x)f_1(\sigma x)f_{n-2}(\sigma^{2}x)e^{2C}\leq ....\leq
\vert\vert f_{1}\vert\vert^{n-j}_{\infty}e^{C(n-j)}f_{j}(\sigma^{i}x).
\end{equation*}
Now let $d_{1}=C_{N_1}, d_{j}=C_{N_2}, 1\leq N_1, N_2\leq m-1$ and  $E_{a^{i}d_{1}\dots d_{j}b^{k}}\in A_{i, j,k}$. Take a point $x\in E_{a^{i}d_{1}\dots d_{j}b^{k}}$ such that
$f_n[E_{a^{i}d_1\dots d_jb^{k}}]\leq 2f_n(x)$. Call it $\bar{x}_{a^{i}d_{1}\dots d_{j}b^{k}}$. Then
$\bar{x}_{a^{i}d_{1}\dots d_{j}b^{k}}\in C_{x_{1}\dots x_{i}N_1y_{2}\dots y_{j-1}N_{2}z_{1}\dots z_{k}}$, where $x_{l}\geq m$ for $1\leq l \leq i$, $z_{l}\geq m$ for $1\leq l \leq k$, $y_{l}\geq 1$ for $2\leq l\leq j-1$.
Therefore,
\begin{align*}
f_n[E_{a^{i}d_1\dots d_jb^{k}}]&\leq 2f_{n}(\bar{x}_{a^{i}d_{1}\dots d_{j}b^{k}}) 
\leq 2\vert\vert f_1\vert\vert ^{n-j}_{\infty}e^{C(n-j)}f_{j}(\sigma^{i}\bar{x}_{a^{i}d_{1}\dots d_{j}b^{k}})  \\
&\leq 2\vert \vert f_1 \vert \vert ^{n-j}_{\infty}e^{C(n-j)}\sup\{f_{j}(x): x\in C_{N_1y_{2}\dots y_{j-1}N_{2}}\}\\
&\leq 2\vert \vert f_1 \vert \vert ^{n-j}_{\infty}e^{C(n-j)}M f_{j}(x),
\end{align*}
for any $x\in C_{N_{1}y_{2}\dots y_{j-1}N_2}$.
Consider a point $N_{2}z=(N_2, z_{0},\dots, z_{n},\dots)\in \Sigma$ and let it be fixed. Denote by $B_{j}(\Sigma)$ the set of admissible words of length $j$ in $\Sigma$.
Then for $d_{1}=C_{N_1}, d_{j}=C_{N_2}$,
\begin{equation*}
\sum_{E_{a^{i}d_{1}\dots d_{j}b^{k}}\in A_{i, j, k}}f_{n}[E_{a^{i}d_1\dots d_jb^{k}}]\leq  2\vert \vert f_1 \vert \vert ^{n-j}_{\infty}e^{C(n-j)}M
\sum_{\substack{x=N_{1}\dots N_{2}z\in \Sigma \\ N_{1}\dots N_{2}\in B_{j}(\Sigma)}}f_{j}(x).
\end{equation*}
By an argument similar to that used to prove Lemma \ref{upper}, we obtain
\begin{equation*}
\sum_{\substack{x=N_{1}\dots N_{2}z\in \Sigma \\ N_{1}\dots N_{2}\in B_{j}(\Sigma)}}f_{j}(x) \leq e^{C(j-1)}\vert\vert f_1\vert\vert_{\infty}L_{f_1}^{j-1} \chi_{C_{N_1}}(N_2z)
\leq e^{C(j-1)}\vert\vert f_1\vert\vert_{\infty}\vert \vert L_{f_1}1\vert\vert^{j-1}_{\infty}.
\end{equation*}
Since $N_{1}, N_{2}\in \{1, \dots, m-1\}$, we have
\begin{eqnarray*}
S_{i, j, k}&=\sum_{E\in A_{i, j, k}} f_n[E]\leq 2\vert\vert f_1\vert\vert^{n-j}_{\infty}e^{C(n-j)}(m-1)^2Me^{C(j-1)}\vert\vert f_1\vert\vert_{\infty}\vert\vert L_{f_1}1\vert\vert_{\infty}^{j-1}.
\end{eqnarray*}
The above inequality also holds for $j=1$. 
For each fixed $n$, consider $j\geq 1$ such that the right hand side of the above inequality takes the maximal value at $j$. Call it $j_n$.
For $j=0 \text{ (and so } i+k=n+1)$, it is easy to see that $S_{i,0,k}\leq 2\vert \vert f_1 \vert \vert _{\infty}^n e^{(n-1)C}$. Let 
\begin{equation*}B_n=\max \{2\vert \vert f_1 \vert \vert _{\infty}^n e^{(n-1)C}, 
2\vert\vert f_1\vert\vert^{n-j_n}_{\infty}e^{C(n-j_n)}(m-1)^2Me^{C(j_n-1)}\vert\vert f_1\vert\vert_{\infty}\vert\vert L_{f_1}1\vert\vert_{\infty}^{j_n-1}\}.
\end{equation*}
Since $1\leq i,k\leq n+1, 0\leq j\leq n-1$
\begin{equation*}
P_n(a, b)=\frac{1}{n} \log \left(\sum_{i,j,k}S_{i,j,k} \right) \leq \frac{1}{n}\log (n+1)^3 B_n.
\end{equation*}
Therefore,
\begin{align*}
&\limsup_{n\rightarrow \infty}P_n(a, b) \\
&\leq \max \left\{\log \vert\vert f_1\vert\vert_{\infty}+\log \vert\vert L_{f_1}1\vert \vert_{\infty} +2C, 2C, \log \vert\vert f_1\vert\vert_{\infty}+2C, \log \vert\vert L_{f_1}1\vert \vert_{\infty}+2C\right\}.
\end{align*}
\end{proof}

For completeness, we will include the final part of the proof of \cite[Theorem 4.4]{s4} (see also \cite[Theorem 3]{s1}). By Lemma \ref{keylema} 
and  (\ref{entropy}),
\begin{align*}
&h(\mu)+\lim_{n\rightarrow \infty}\frac{1}{n}\int \log f_{n}d\mu \leq \limsup_{n \rightarrow \infty} \left(\sum_{a,b\in \beta}\mu(a\cap \sigma^{-n}b)P_n(a, b)\right)\\
&\leq \limsup_{n\rightarrow \infty}\left(P(\F)\sum_{a,b\neq C_{\geq m}}\mu(a\cap \sigma^{-n}b)+C'\sum_{a=C_{\geq m} \text{ or } b=C_{\geq m}}\mu(a\cap \sigma^{-n}b)\right)
\\ & \leq P(\F)+C'(\mu(C_{\geq m})+\mu(\sigma^{-n}(C_{\geq m}))) \leq P(\F)+2C'\mu(C_{\geq m}).
\end{align*}
Letting $m\rightarrow \infty$, we obtain the result.
\end{proof}

The set of $\sigma$-invariant Borel probability measures, $\M$,  is a very large and complicated convex, non-compact set. Indeed, it strictly contains a countable family of Poulsen simplexes, that is, infinite dimensional compact and convex sets with the property that the extreme points are dense in the set. It is therefore a major problem in the ergodic theory of countable Markov shifts that  of choosing relevant invariant measures. The variational principle provides a criteria for making that choice.

\begin{defi}
Let  $(\Sigma, \sigma)$ be a topologically mixing countable state Markov shift  and  $\F=\{\log f_n\}_{n=1}^{\infty}$ be an almost-additive  sequence on $\Sigma$. A measure $\mu \in \M$ is said to be an \emph{equilibrium measure} for $\F$ if
\begin{equation*}
P(\F)= h(\mu) + \lim_{n \to \infty} \frac{1}{n} \int \log f_n \ d \mu.
\end{equation*}
\end{defi}

\section{Gibbs measures} \label{sg}
In this section, we prove the existence of Gibbs measures for an almost-additive sequence of continuous functions under certain assumptions. In order to do so, we require an additional assumption on the combinatorial structure of the Markov shift. This is a necessary assumption  
in the classical thermodynamical formalism for countable Markov shifts (see \cite{mu,s3}). Let us start with some basic definitions.

\begin{defi} \label{def-gibbs}
Let  $(\Sigma, \sigma)$ be a topologically mixing countable state Markov shift  and  $\F=\{\log f_n\}_{n=1}^{\infty}$ be an almost-additive  sequence on $\Sigma$. A measure $\mu \in \M$ is said to 
be \emph{Gibbs} for $\F$ if there exist constants $C>0$ and $P \in \R$ such that for every $n \in \N$ and every $x \in C_{i_0 \dots i_{n-1}}$ we have
\begin{equation*}
\frac{1}{C} \leq	\frac{\mu(C_{i_0 \dots i_{n-1}})}{\exp(-nP)f_n(x)}	\leq C.
\end{equation*}
 \end{defi}

There is a special class of Markov shifts having a combinatorial structure  similar to that of the full-shift, that will be important for us.
\begin{defi} \label{BIP}
A countable Markov shift  $(\Sigma, \sigma)$ is said to satisfy  the \emph{big images and preimages property (BIP property)} if 
there exists $\{ b_{1} , b_{2}, \dots, b_{n} \}$ in the alphabet $S$ such that
\begin{equation*}
\forall a \in S \textrm{ } \exists i,j \textrm{ such that } t_{b_{i}a}t_{ab_{j}}=1.
\end{equation*}
\end{defi}

It was shown by Mauldin and Urba\'nski \cite{mu} and also by Sarig  \cite{s3} that there is a combinatorial obstruction to the existence of Gibbs measures corresponding to a 
continuous function of summable variations. Indeed, if $(\Sigma, \sigma)$ is a topologically mixing countable Markov shift that does not satisfy the BIP property, then no continuous function 
can have Gibbs measures. It should also be noticed that in the compact setting of  Markov shifts over a finite alphabet,  Gibbs measures are always equilibrium measures. This is 
no longer true in the non-compact setting of countable Markov shifts. Indeed, a Gibbs measure $\mu$ for a continuous function $\phi$ could satisfy 
$h(\mu)= \infty$ and $\int \phi \ d \mu= - \infty$.  In such a situation, the measure $\mu$ is not an equilibrium measure for $\phi$ (see \cite{s3} for comments and examples). 
Of course, this type of phenomena can also occur in our context. 

Note that if  Let $\F=\{\log f_n\}_{n=1}^{\infty}$ is an almost-additive Bowen 
sequence on $\Sigma$ with  $\sum_{a\in S}\sup f_1\vert _{C_a}<\infty$,  then  implies  $\vert \vert L_{f_1}1 \vert \vert _{\infty}<\infty$. In particular the pressure is finite, $P(\F)<\infty$, and the variational principle (see Theorem \ref{main1}) holds.  The main result of this section is the following

\begin{teo}\label{main2}
Let $(\Sigma, \sigma)$ be a topologically mixing countable state Markov shift with the BIP property. Let $\F=\{\log f_n\}_{n=1}^{\infty}$ be an almost-additive Bowen 
sequence on $\Sigma$ with  $\sum_{a\in S}\sup f_1\vert _{C_a}<\infty$.  Then  there is a Gibbs measure $\mu$ for $\F$ and it is mixing. Moreover, If $h(\mu)<\infty$, then it is the unique equilibrium measure for $\F$. 
\end{teo}

\begin{proof}
The  proof is  inspired on \cite[Lemma 2.8]{mu} and \cite[Lemmas 1, 2 and Theorem 5]{b2}. Those results need to be modified and adapted to the almost-additive
setting and to the case of a non-compact phase space.  We identify a countable alphabet S with the set $\N$. 

Since $(\Sigma, \sigma)$ is topologically mixing and has the BIP property, there exist $k\in \N$ and a finite collection 
$W$ of admissible words of length $k$ such that for any $a, b\in S$, there exists $w\in W$ such that $awb$ is admissible (see \cite[p.1752]{s3} and \cite{mu}). Let $A$ be the transition matrix for $\Sigma$. 
By rearranging the set $\N$, there is an increasing sequence $\{l_{n}\}_{n=1}^{\infty}$ such that the matrix $A\vert_{ \{1, \dots, l_{n}\}\times \{1, \dots, l_{n}\}}$ 
is primitive. Let $Y_{l_n}$ be the topologically 
mixing finite state Markov shift with the transition matrix $A\vert_{ \{1, \dots, l_{n}\}\times \{1, \dots, l_{n}\}}$. Then there exists $p\in \N$ such that 
for all $n\geq p$, $Y_{l_n}$ contains all admissible words in  $W$.  We denote by $B_n(Y_l)$ the set of admissible words of length $n$ in $Y_l$.
Since $\F$ is almost-additive, there exists $C>0$ such that for each $n, m\in \N, x\in\Sigma$,
\begin{equation}\label{e1}
f_n(x) f_{m}(\sigma^n x) e^{-C} \leq f_{n+m}(x)\leq  f_n(x) f_{m}(\sigma^n x) e^{C}.
\end{equation}
\begin{claim}\label{claim}  
For $Y_{l_n} \subset \Sigma$, $n\geq p$, there is a unique equilibrium measure for $\F\vert_{Y_{l_n}}$ 
and it is Gibbs for $\F\vert_{Y_{l_n}}$. Moreover, the constant $C$ (see Definition \ref{def-gibbs}) can be chosen in such a way that $C$ is independent of $Y_{l_n}$.
\end{claim}
\begin{proof} [Proof of the Claim] Since $\F$ is a Bowen sequence, we assume that it satisfies equation (\ref{bowenbound}).  Clearly, $\F\vert_{Y_{l_n}}$ is 
an almost-additive Bowen sequence on $(Y_{l_n}, \sigma\vert_{Y_{l_n}})$. Therefore, the first part of the claim is 
immediate by \cite[Theorem 5]{b2} or  \cite[Theorem 6]{m}. Slightly modifying the proof 
of  \cite[Lemmas 1, 2 and Theorem 5]{b2}, we will obtain the second part of the claim. 
We only show the first step of the proof to see how we can get a uniform constant $C$. By the assumptions, any admissible word in $W$ is an admissible word in $Y_{l_n}$ for all  
$n\geq p$. Fix $Y_{l_n}$, $n\geq p$, and call it $Y$.  Define $\alpha_n^{Y}=\sum_{i_0\cdots i_{n-1}\in B_{n}(Y)}\sup\{f_n\vert _Y(y): y\in C_{i_0\dots i_{n-1}}\}$.
For $l \in \N$, let $\nu_{l}$ be the Borel probability measure on $Y$ defined by  
\begin{equation*}
\nu_{l}(C_{i_0\dots i_{l-1}})=\frac{\sup\{f_{l}\vert_{Y}(y):y\in C_{i_{0}\dots i_{l-1}}\}}{\alpha_{l}^{Y}}.
\end{equation*}
Let $n\in \N$ and $l\geq n+k$.  For any admissible words 
$i_{0}\cdots i_{n-1}$ and $j_0\dots j_{l-k-1}$ in $Y$, there exists $m_0\dots m_{k-1}\in W$ such that $i_0 \dots i_{n-1} m_0 \dots m_{k-1}j_{0}\dots j_{l-k-1}$ 
is an admissible word in $Y$. 
For $w\in W$, let $N_{w}=\sup\{f_k(z):z\in C_w\}$ and 
$\bar{N}=\min\{N_{w}: w\in W\}$. For any $y=(y_0, \dots, y_n,\dots)\in Y\subset \Sigma$ with $y_{0}\dots y_{k-1}=w\in W$, we have 
\begin{equation*}\label{k1}
\frac{N_{w}}{f_{k}\vert_{Y}(y)}\leq M. 
\end{equation*}
By (\ref{e1}) and (\ref{k1}), for each 
$y\in C_{i_0\dots i_{n-1} m_0\dots m_{k-1}j_{0}\dots j_{l-k-1}}$, we have
\begin{align*}
f_{l+n}\vert_Y(y) &\geq f_{n}\vert_Y(y)f_{k}\vert_Y(\sigma^{n}y)f_{l-k}\vert_Y(\sigma^{n+k}y)e^{-2C}\\
&\geq \frac{\bar{N}e^{-2C}}{M^3} \sup \{f_n\vert_Y(y):y\in C_{i_0 \dots i_{n-1}}\}\sup\{f_{l-k}\vert_Y(y): y\in C_{j_0 \dots j_{l-k-1}}\}.  
\end{align*}
For each fixed $i_0\dots i_{n-1}\in B_{n}(Y)$, we have 

\begin{align}\label{lowerbound}
&\sum_{t_{0}\dots t_{l-1}}\sup \{f_{n+l}\vert_Y(y): y\in C_{i_{0}\dots i_{n-1} t_{0}\dots t_{l-1}}\}\\
&\geq \sum_{j_{0}\dots j_{l-k-1}}\frac{\bar{N}e^{-2C}}{M^{3}}\sup\{f_{n}\vert_Y(y):y\in C_{i_{0}\dots i_{n-1}}\} \sup \{f_{l-k}\vert_Y(y):y\in C_{j_{0}\dots j_{l-k-1}} \}\\
&\geq \frac{\bar{N}e^{-2C}}{M^{3}} \sup\{f_n\vert_Y(y): y\in C_{i_{0}\dots i_{n-1}} \} \alpha_{l-k}^{Y}.
\end{align}
Since $\alpha_{l}^{Y}\leq e^{C}\alpha_{n}^{Y}\alpha_{l-n}^{Y}$, we obtain 
\begin{equation*}
\alpha_{n+l}^{Y} \geq \frac{\bar{N}e^{-3C}\alpha_n^{Y} \alpha_{l}^{Y}}{M^{3}\alpha_{k}^{Y}}.
\end{equation*}
We note that by (\ref{e1})
\begin{equation*}
\alpha_{k}^{Y}=\sum_{i_{0}\dots i_{k-1}} \sup\{f_k\vert_Y(z):z\in C_{i_{0}\dots i_{k-1}}\}\leq e^{(k-1)C}(\sum_{i\in \N}f_1\vert_{C_i})^{k}<\infty.
\end{equation*}
Therefore, there exists $C_1>0$ such that
\begin{equation}\label{p3}
\alpha_{n+l}^{Y}\geq C_{1}\alpha_n^{Y}\alpha_l^{Y},
\end{equation}
and clearly $C_1$ does not depend on $l_n$.
Now we observe that 
\begin{equation}\label{p}
P(\F\vert _{Y})=\lim_{n \rightarrow\infty}\frac{1}{n}\log \alpha_n^{Y}.
\end{equation}
To see this, we claim that there exists a constant $C_2>0$ such that for a symbol $a\in \N$,
\begin{equation}\label{p4}
C_2 \alpha_n^{Y} \leq \sum_{z\in C_{a}, \sigma^{n+2k+1}z=z}f_{n+2k+1}\vert_Y(z).
\end{equation}
Clearly,
\begin{equation*}
f_{n+2k+1}\vert_Y(z)\geq f_{k+1}\vert_Y(z)f_{n}\vert_Y(\sigma^{k+1}z)f_{k}\vert_Y(\sigma^{n+k+1}z)e^{-2C}.
\end{equation*}
Therefore, 
\begin{align*}
\sum_{z\in C_a, \sigma^{n+2k+1}z=z}f_{n+2k+1}\vert_Y(z) & \geq \frac{{\bar{N}}^2e^{-2C}}{M}\sup\{f_{1}\vert_Y(z):z\in 
C_{a}\} \sum_{z\in C_{z_0 \dots z_{n-1}}}f_n\vert_Y(z)\\
&\geq \frac{{\bar{N}}^2e^{-2C}}{M^2} \sup\{f_{1}\vert_Y(z):z\in 
C_{a}\} \alpha_n^{Y},
\end{align*}
where the summation in the right hand side of the first inequality is taken over a set containing a point $z$ from each cylinder set $C_{z_0 \dots z_{n-1}}$.
This proves (\ref{p4}). We also note that from the construction $C_2>0$ does not depend on $l_n$. 
By the definition of the pressure, we obtain  (\ref{p}). Since 
\begin{equation*}
\lim_{n\rightarrow\infty}\frac{1}{n}\log \alpha_{n}^{Y}=\inf_{n\in \N}\frac{1}{n}\log e^{C} \alpha_n^{Y}, 
\end{equation*}
we have $e^{C}\alpha_n^{Y}\geq e^{P(\F\vert_Y)n}$.
Similarly, by (\ref{p3}),  
\begin{equation*}
P(\F\vert _{Y})=\sup_{n\in \N}\frac{1}{n}\log C_{1}\alpha_n^{Y}.
\end{equation*}
Therefore, 
\begin{equation}\label{bpressure}
 C_1\alpha_n^{Y}\leq e^{P(\F\vert_Y)n}\leq e^{C}\alpha_n^{Y}.
\end{equation}
Using (\ref{lowerbound}), (\ref{p3}) and  (\ref{bpressure}),  similar arguments to those in \cite{b2} show 
that there exist constants $C_3, C_4>0$, both independent of $l_n$, such that for all $n\in \N$
\begin{equation}\label{gibbsp}
 C_3\leq \frac{\nu_{l}(C_{i_0 \dots i_{n-1}})}{e^{-nP(\Phi(\F\vert_{Y}))}f_n\vert_{Y}(y)}\leq C_4, i_0 \dots i_{n-1}\in B_{n}(Y).
\end{equation}
Now we take a convergent subsequence $\{\nu_{l_k}\}_{k=1}^{\infty}$ of  $\{\nu_{l}\}_{l=1}^{\infty}$ and let $\nu$ be the limit point. 
Then $\nu$ 
satisfies equation (\ref{gibbsp}), (replacing $\nu_l$ by $\nu$). Modifying arguments in \cite{b2} by using the property of 
bounded variation and the BIP property as seen in the above arguments, we conclude 
that we can choose a constant $C$ (in Definition \ref{def-gibbs}) for the $\sigma_{Y}$-invariant ergodic Gibbs measure $\mu_Y$ for $F\vert_Y$ such that $C$ is independent of $l_n$.
\end{proof}
By the claim above, for each fixed $l_n$,$n\geq p$, $\F\vert_{Y_{l_n}}$ has a unique equilibrium state $\mu_{l_n}$ which is Gibbs, i.e., for each $q\in \N$, there exist $\tilde{C_1}, \tilde {C_2}>0$ such that
\begin{equation}\label{eachgibbs}
\tilde C_1 \leq \frac{\mu_{l_n}(C_{i_0\cdots i_{q-1}})}{e^{-qP(\F\vert_{Y_{l_n}})}f_{q}\vert_{Y_{l_n}}(y)} \leq \tilde{C_2},\text{ for } y\in C_{i_0\dots i_{q-1}}, i_{0}
\dots i_{q-1}\in B_q(Y_{l_n}).
\end{equation}
We first show that the sequence $\{\mu_{l_n}\}_{n=p}^{\infty}$ of $\sigma$-invariant Borel probability measures on $\Sigma$ is tight.  Let $\pi_{k}:\Sigma\rightarrow \N$ be the projection onto the $k$-th coordinate. 
Then for each $a\in \N$,
\begin{align*}
\mu_{l_n}(\pi^{-1}_k(a)) &= \sum_{i_0 \dots i_{k-2}a\in B_k(Y_{l_n})}\mu_{l_n}(C_{i_0\cdots i_{k-2}a}) \\
& \leq  \tilde{C_2}\sum_{i_0 \dots i_{k-2}a \in B_k(Y_{l_n}), y\in C_{i_0\dots i_{k-2}a}}e^{-kP(\F\vert_{Y_{l_n}})}f_{k}\vert_{Y_{l_n}}(y) (by  (\ref{eachgibbs})),
\end{align*}
where the summation is taken over a set containing a point $y$ from each cylinder set $C_{i_0\dots i_{k-2}a}$.
By (\ref{e1}),  
\begin{equation*}
f_{k}\vert_{Y_{l_n}}(y)\leq f_{k-1}\vert_{Y_{l_n}}(y)f_{1}\vert_{Y_{l_n}}\left(\sigma^{k-1}y\right)e^C\leq \left(\sup f_{1}\vert_{C_a}\right)f_{k-1}\vert_{Y_{l_n}}(y)e^C.
\end{equation*}
Thus
\begin{align*}
&\sum_{i_0\dots i_{k-2}a\in B_k(Y_{l_n}), y\in C_{i_0\dots i_{k-2}a}}f_{k}\vert_{Y_{l_n}}(y) \\
&\leq 
\left(\sup f_{1}\vert_{C_a} \right)e^C \sum_{i_0\dots i_{k-2}a\in B_k(Y_{l_n}), y\in C_{i_0\dots i_{k-2}a}} f_{k-1}\vert _{Y_{l_n}}(y)\\
& \leq \left(\sup f_{1}\vert_{C_a} \right)e^{C(k-1)}\sum_{i_0 \dots i_{k-2}a\in B_{k}(Y_{l_n})}\left(f_{1}\vert_{Y_{l_n}}(y)f_{1}\vert _{Y_{l_n}}(\sigma y)\dots 
f_1 \vert_{Y_{l_n}}(\sigma^{k-2}y) \right)\\
&\leq \left(\sup f_{1}\vert _{C_a}\right)e^{C(k-1)}\sum_{i_{0}\dots i_{k-2}a\in B_{k}(Y_{l_n})}\sup f_{1}\vert_{C_{i_0}}\dots \sup f_{1}\vert_{C_{i_{k-2}}}\\
& \leq \left(\sup f_{1}\vert _{C_a}\right)e^{C(k-1)}\left(\sum_{i\in \N}\sup f_{1}\vert_{C_i}\right)^{k-1}.
\end{align*}
Hence
\begin{eqnarray*}
\mu_{l_n}(\pi^{-1}_k(a))&\leq \tilde{C_2} e^{-kP(\F\vert _{Y_{l_n}})}\left(\sup f_1\vert_{C_a}\right)e^{C(k-1)}(\sum_{i\in \N}\sup f_{1}\vert_{C_i})^{k-1}.
\end{eqnarray*}
Now set  $N=P(\F\vert_{Y_{l_1}})$ if $P(\F)<0$, and $N=\min\{-P(\F), -\vert P(F\vert_{Y_{l_1}})\vert \}$ otherwise. Then we obtain
\begin{eqnarray*}
\mu_{l_n}(\pi_k^{-1}[a+1, \infty))\leq \tilde{C_2} e^{-kN+C(k-1)}\left(\sum_{i\in \N}\sup f_{1}\vert_{C_i}\right)^{k-1}
\sum_{i>a}\sup f_1\vert_{C_i}.
\end{eqnarray*}
Since  $\sum_{i\in \N}\sup f_{1}\vert_{C_i}<\infty$, for each given $\epsilon>0, k\in \N$, we can find $n_k \in \N$ 
with the property of
\begin{equation*}
\tilde{C_2} e^{-kN+C(k-1)}\left(\sum_{i\in \N}\sup f_{1}\vert_{C_i}\right)^{k-1}
\sum_{i>n_k}\sup f_1\vert_{C_i}\leq \frac{\epsilon}{2^k}.
\end{equation*}
Therefore, for any $l_n, k \in \N$, 
\begin{equation*}
\mu_{l_n}\left(\pi_k^{-1}[n_k+1, \infty) \right)\leq\frac{\epsilon}{2^k},
\end{equation*}
and so
\begin{equation*}
\mu_{l_n}\left(\Sigma\cap \prod_{k\geq 1}[1, n_k]\right)\geq 1-\sum _{k\geq 1}\mu_{l_n}\left(\pi_{k}^{-1}([n_k+1, \infty))\right)\geq 1-\epsilon.
\end{equation*}
Since $\Sigma \cap \prod_{k\geq 1}[1, n_k]$ is a compact subset of $\Sigma$, by Prohorov's theorem, the sequence $\{\mu_{l_n}\}_{n=p}^{\infty}$ is tight.
Therefore, there exists a convergent subsequence $\{\mu_{{l}_{n_k}}\}_{k=1}^{\infty}$ of $\{\mu_{l_n}\}_{n=p}^{\infty}$. We denote by $\mu$ a limit point of this subsequence. Since it is a limit point 
of a sequence of invariant measures on $\Sigma$, $\mu$ is also $\sigma$-invariant on $\Sigma$. By the property (\ref{eachgibbs}), letting $l_n\rightarrow \infty$, we obtain
for $q\in \N$ and each $ y\in C_{i_0\dots i_{q-1}}, i_0\dots i_{q-1}\in B_q(\Sigma)$,
\begin{equation}\label{g}
\tilde C_1 \leq \frac{\mu\left(C_{i_0\cdots i_{q-1}}\right)}{e^{-qP(\F)}f_{q}(y)} \leq \tilde{C_2}.
\end{equation}
Therefore, $\mu$ is a Gibbs measure for $\F$.

In order to show that $\mu$ is ergodic, we use similar arguments to those used to prove \cite[Lemma 2]{b2}.
Let $i_0\dots i_{n-1}$ and $j_{0}\dots j_{l-1}$ be fixed admissible words in $\Sigma$. Let $k$ and $\bar{N}$ be defined as in the proof of claim \ref{claim}. We also define 
$\alpha_{n}^{\Sigma}$ in a similar manner as we defined $\alpha_{n}^{Y}$. Then for $m-n\geq k$,
\begin{align*}
&\mu(C_{i_{0}\dots i_{n-1}}\cap f^{-m}(C_{j_{0}\dots j_{l-1}}))\\
&=\sum_{i_{0}\dots i_{n-1}k_{0}\dots k_{m-n-1}j_{0}\dots j_{l-1}\in B_{m+l}(\Sigma)}\mu(C_{i_{0}\dots i_{n-1}k_{0}\dots k_{m-n-1}j_{0}\dots j_{l-1}})\\
&\geq \tilde{C_1} e^{-(m+l)P(\F)} \sum_{y\in C_{i_0 \dots i_{n-1}k_{0}\dots k_{m-n-1}j_{0}\dots j_{l-1}}}f_{m}(y)f_{l}(\sigma^m{y})f_{m-n}(\sigma^{n}y)e^{-2C}\\
&\geq \frac{e^{-2C}\tilde{C_1}e^{-(m+l)P(\F)}}{M^2}\sup\{f_n(y):y\in C_{i_{0}\dots i_{n-1}}\}
\sup\{f_{l}(y):y\in C_{j_{0}\dots j_{l-1}}\}\\
&\sum_{y\in C_{i_{0}\dots i_{n-1}k_{0}\dots k_{m-n-1}j_{0}\dots j_{l-1}}}f_{m-n}(\sigma^ny),
\end{align*}
where in the second and third inequalities each summation is taken over a set containing a point $y$ from each cylinder set 
$C_{i_{0}\dots i_{n-1}k_{0}\dots k_{m-n-1}j_{0}\dots j_{l-1}}$ in $\Sigma$.
Then
\begin{align*}
&\sum_{y\in C_{i_0\dots i_{n-1}k_{0}\dots k_{m-n-1}j_{0}\dots j_{l-1}}}f_{m-n}(\sigma^ny)\\
&\geq e^{-2C}\sum _{y\in C_{i_0 \dots i_{n-1}k_{0}\dots k_{m-n-1}j_{0}\dots j_{l-1}}}f_{k}(\sigma^n y)f_{m-n-2k}(\sigma^{n+k}y)f_{k}(\sigma^{m-k}y)\\
&\geq \frac{e^{-2C}\bar{N}^2}{M^3}\alpha_{m-n-2k}^{\Sigma}.
\end{align*}
Note that we have $e^{C}\alpha_{n}^{\Sigma}\geq e^{P(\F)n}$ by the proof used in the proof of the claim \ref{claim}. Therefore, using  (\ref{g}), 
it is easy to see that there exits a constant $\tilde{C}_3>0$ such that 
\begin{equation*}
\mu(C_{i_{0}\dots i_{n-1}}\cap f^{-m}(C_{j_{0}\dots j_{l-1}}))\geq \tilde{C}_3 \mu(C_{i_{0}\dots i_{n-1}})\mu(C_{j_{0}\dots j_{l-1}}).
\end{equation*}
Therefore, $\mu$ ergodic and thus it is the unique Gibbs measure for $\F$. If $h(\mu)<\infty$, using the proof of \cite[Theorem 3.5]{mu}(replace $S_nf$ and $P(f)$ by $f_n$ and $P(\F)$ respectively), 
it is the unique equilibrium measure for $\F$.
The fact that the measure $\mu$ is mixing is fairly standard and follows as in the final part of the proof of  \cite[Theorem 5]{b2}. 
\end{proof}

\section{Example 1: The full-shift} \label{ex1}
Let us consider the full-shift on a countable alphabet, that is $(\Sigma_F, \sigma)$, where
\[ \Sigma_F:=\{ (x_i)_{i=0}^{\infty} : x_i \in \N \} . \]
The good combinatorial properties of this shift allow us to make some explicit computations. Indeed, let $\{\lambda_j\}_{j=1}^{\infty}$ be a sequence of real numbers such that 
$\lambda_j \in (0,1)$ and $\sum_{j=1}^{\infty} \lambda_j < \infty$. Let   $\{\log c_n\}_{n=1}^{\infty}$ be an almost-additive  sequence of real numbers, that is, there exists a constant $C>0$ such that  
\[ e^{-C}c_n c_m \leq  c_{n+m} \leq  e^Cc_n c_m.\]
For $n\in \N$, define $f_n:\Sigma_F\rightarrow \R$ by  
\begin{equation*}
 f_n(x)=  c_n \lambda_{i_0} \lambda_{i_1} \cdots \lambda_{i_{n-1}}, \text{ for } x \in C_{i_0 \dots i_{n-1}}.
\end{equation*}
Let $\F=\{ \log f_n \}_{n=1}^{\infty}$. Then $\F$ is an almost-additive Bowen sequence on $\Sigma_F$. By definition we have
\begin{align*}
P(\F)&=\lim_{n \to \infty} \frac{1}{n} \log \left(\sum_{i_0, i_1, \dots, i_{n-1} \in \N} c_n\lambda_{i_0} \lambda_{i_1} \cdots \lambda_{i_{n-1}}\right) \\
&=\lim_{n \to \infty} \frac{\log c_ n}{n}  + \log \left(\sum_{i=1}^{\infty} \lambda_i \right).
\end{align*}
Clearly, Proposition \ref{approx} and the variational principle (Theorem \ref{main1}) hold for $\F$. Since $\sum_{i\in \N}\sup f_1\vert_{C_{i}}=c_1\sum_{i\in \N}\lambda_i<\infty$, 
by Theorem \ref{main2}, there exist a Gibbs measure $\mu$ for $\F$ which is mixing. If $h(\mu)<\infty$, then it is the unique equilibrium measure for $\F$. 

Similarly, we obtain an explicit formula for the pressure function for $t\F=\{ t\log f_n \}_{n=1}^{\infty}$, with $t \in \R$, namely
\begin{eqnarray*}
P(t\F)=\lim_{n \to \infty} \frac{t\log c_n}{n}   +  \log \left (\sum_{i=1}^{\infty} {\lambda_i}^t\right).
\end{eqnarray*}
In particular, there exists $t'>0$ such that 
\begin{equation*}
P(t\F)=
\begin{cases}
\infty & \text{ if } t<t';\\
\text{finite,} & \text{ if } t >t'.
\end{cases}
\end{equation*}
Moreover, for $t>t'$ the pressure function $t \to P(t\F)$  is real analytic, convex and decreasing. 

\section{Example 2: A factor map} \label{ex2}
Factor maps between subshifts on a finite alphabet have been studied using equilibrium measures for continuous functions, 
subadditive-sequences of continuous functions and almost-additive sequences of continuous functions (see \cite{SS, BP, Fe4,Y}). The present example exhibits one of the pathologies that one might encounter when trying to extend the factor map theory to the countable Markov shift setting. Indeed,  we construct an almost-additive sequence of continuous functions, $\F$,  on a countable  Markov shift by considering a one-block factor map between countable Markov shifts. These type of sequence of functions provide relevant information about the factor map in the case of finite state sub-shift of finite type. Nevertheless in our setting the pressure is infinite, $P(\F)=\infty$. 
Let $A$ be the matrix defined by 
$A=(a_{ij})_{\N_{0}\times \N_{0}}$ where $a_{i0}=a_{0j}=1$, for all $i,j\in \N_{0}$, and $a_{ij}=0$ otherwise. 
Let $B$ be the matrix defined by $B=(b_{ij})_{\N\times \N}$, $b_{i1}=b_{1j}=1$, for $i,j\in \N$, and $b_{ij}=0$ otherwise.
Let $X, Y$ be the topologically mixing countable Markov shifts with the BIP property 
determined by the transition matrix $A,B$ respectively. Let $\pi:X\rightarrow Y$ 
be a factor map defined by $\pi(k)={k}/{2}+1$ if $k$ is even, and $\pi(k)=(k-1)/{2}+1$ if $k$ is odd. 
   For an admissible word $y_1\dots y_n$ of length $n$ on Y, denote by 
$\vert \pi^{-1}[y_1\dots y_n]\vert$ the 
number of admissible words of length $n$ in $X$ that are mapped to $y_1\dots y_n$ by $\pi$.
Define $\phi_n:Y \rightarrow \R$ by $\phi_n(y)=\vert \pi^{-1}[y_1\dots y_n]\vert$ and let
$\F=\{-\log \phi_n\}_{n=1}^{\infty}$. Let $A_1$ be the transition matrix of the symbols of $\pi^{-1}\{1\}=\{0,1\}$ and $a$ be the largest eigenvalue of the matrix $A_1$. 
It is easy to see that 
in general a point $y$ in $Y$ is given by $y=1^{n_1}i_11^{n_2}i_2\dots$ or $y=i_11^{n_2}i_2\dots$, where $i_1, i_2, \dots\geq 2,  n_1, n_2, \dots \geq 1$ and that there exist $C_1, C_2>0$ such that for any $l\in \N$,
\begin{equation*} 
C_12^{l}a^{n_1+\cdots+n_l}\leq \vert \pi^{-1}[1^{n_1}i_11^{n_2}\dots i_{l-1}1^{n_l}i_{l}]\vert \leq C_22^{l}a^{n_1+\cdots+n_l}.
\end{equation*}
We can approximate the upper bouds and lower bounds of  $\vert \pi^{-1}[1^{n_1}i_11^{n_2}\dots i_{l-1}1^{n_l}]\vert$, $\vert \pi^{-1}[i_11^{n_2}\dots i_{l-1}1^{n_l}i_{l}]\vert$, and    
$\vert \pi^{-1}[i_11^{n_2}\dots i_{l-1}1^{n_l}]\vert$ similary. 
Therefore, $\F$ is an almost-additive Bowen sequence on $Y$ and Proposition \ref{approx} holds. However, since the entropy of $(Y,\sigma)$ is infinite and the functions $\frac{1}{n} \log \phi_n$ are uniformly bounded (below and above)  we have that $P(\F)=\infty$.

\section{Maximal Lyapunov exponents of product of matrices} \label{le}
Let $A, B$ be two square matrices of size $d \times d$. Let $U$ be the $d-$dimensional column vector having each coordinate equal to $1$. Consider the following norm
\begin{equation}
\Vert A \Vert := U^t A U.
\end{equation}
Let  $\{A_1, A_2, \dots \}$ be a countable collection of $d \times d$ matrices and  let $(\Sigma, \sigma)$ be a topologically mixing countable Markov shift. 
If $w=(i_0, i_1, \dots) \in \Sigma$,  define the sequence of functions by
\[\phi_n(w)= \Vert A_{i_{n-1}}  \cdots A_{i_1} A_{i_0}\Vert.\]
Since
\begin{equation*}
\Vert AB \Vert \leq \Vert A  \Vert \Vert B \Vert,
\end{equation*}
the sequence $\F= \{\log\phi_n \}_{n=1}^{\infty}$ is sub-additive on $\Sigma$. The study of this type of functions began with the work of Bellman \cite{be} and flourished with the seminal work  of Furstenberg and Kesten \cite{fk} who in 1960 considered the case of finitely many square matrices $\{A_1, A_2, \dots ,A_m\}$ and the full-shit on $m$-symbols. 
They proved that if $\mu \in\M$ is ergodic then $\mu$-almost everywhere the following equality holds:
\begin{equation*}
\lim_{n\to \infty} \frac{1}{n} \int \log \phi_n \ d\mu =  \lim_{n\to \infty} \frac{1}{n} \log \phi_n(w).
\end{equation*}
Kingman \cite{ki}, eight years later, proved his famous sub-additive ergodic theorem from which the above result follows.  The number
\[\lambda(w) := \lim_{n\to \infty} \frac{1}{n} \log \phi_n(w),\]
is called \emph{Maximal Lyapunov exponent of $w$}, whenever the limit exists.  It is a fundamental dynamical quantity whose study arises in a wide range of different 
context, e.g. Schr\"odinger operators \cite{aj}, smooth cocycles \cite{av}, Hausdorff dimension of measures \cite{Fe3}. Actually, its effective computation is also of interest \cite{po}.
Recently, in a series of papers Feng  \cite{Fe1, Fe2, Fe3}  studied dimension theory and thermodynamic formalism for maps $M :Ê\Sigma \to L(\R^d, \R^d)$, where $L(\R^d, \R^d)$
denotes the space of $d \times d$ matrices. The techniques developed in the previous section allow us to generalise some of the results obtained by Feng to this non-compact setting.

In this generality, the sequence $\{\log \phi_n\}_{n=1}^{\infty}$ is only sub-additive (not necessarily almost-additive) and the shift $(\Sigma, \sigma)$ need not to satisfy the BIP property. 
Under certain additional assumptions, we are able to prove the existence of Gibbs measures.

\begin{prop} \label{mat}
Let $(\Sigma, \sigma)$ be a countable Markov shift satisfying the BIP condition. Let  $\{A_1, A_2, \dots \}$ be a countable collection of $d \times d$ matrices 
having strictly positive entries. For $n\in \N$, define $\phi_n: \Sigma  \rightarrow \R$ by \[\phi_n(w)= \Vert A_{i_{n-1}}  \cdots A_{i_1} A_{i_0} \Vert\] for 
$w=(i_{0}, i_{1}, \dots)\in \Sigma$. 
If $\F=\{\log \phi_n\}_{n=1}^{\infty}$ is an almost-additive sequence on $\Sigma$ with  $\sum_{i=1}^{\infty} \Vert A_i \Vert < \infty$, 
then 
\begin{equation*}
P(\F)=\sup \left\{h(\mu)+\lim_{n\rightarrow \infty}\frac{1}{n}\int \log \phi_n \  d\mu:  \mu\in \M       
 \ and \lim_{n\rightarrow \infty}\frac{1}{n}\int \log \phi_n d\mu\neq -\infty \right\},
\end{equation*}
and there exists a Gibbs measure $\mu$ for $\F$ which is mixing.
\end{prop}

\begin{proof}
Note that the continuous functions $\phi_n$ are locally constant over cylinders of length $n$. Therefore, the Bowen condition is satisfied.  Moreover,
\begin{equation*}
\sum_{a\in S}\sup {\phi_1}\vert _{C_a}=\sum_{i=1}^{\infty} \Vert A_i \Vert < \infty.
\end{equation*}
Since the system satisfies the BIP condition, the result follows from Theorem \ref{main2}. 
\end{proof}

\begin{rem}
If the measure $\mu$ in Proposition \ref{mat} is such that $h(\mu) < \infty$, then $\mu$ is the unique equilibrium measure for $\F$. 
\end{rem}

The assumption $\sum_{i=1}^{\infty} \Vert A_i \Vert < \infty$  implies that
$\lim_{n \to \infty}  \Vert A_n \Vert =0.$ Therefore, almost-additivity cannot be obtained (in general)  as in \cite[Lemma 2.1]{Fe1}, even if all the entries are positive. Nevertheless, the same  proof  obtained by Feng in \cite[Lemma 2.1]{Fe1} gives us
\begin{lema}
Let  $\{A_1, A_2, \dots \}$ be a countable collection of $d \times d$ matrices having strictly positive entries. 
Suppose there exists a constant $C>0$ with the property that for every $k \in\N$ the following holds:
\begin{equation*}
\frac{\min_{i,j} (A_k)_{i,j}}{\max_{i,j} (A_k)_{i,j}} \geq dC.
\end{equation*} 
Then the sequence $\F=\{\log \phi_n\}_{n=1}^{\infty}$ is almost-additive on $\Sigma$, where $\phi_n: \Sigma\rightarrow \R$ is defined as in Proposition \ref{mat}  . 
\end{lema}
The above condition is related to the cone condition studied by Barreira and Gelfert in Ê\cite{bg}. Alternative conditions ensuring almost-additivity of $\F$ have been obtained by Feng \cite[Proposition 2.8]{Fe3} and, in a slightly different setting,  by Falconer and Sloan \cite[Corollary 2.3]{fs}.

\begin{rem}
Under the assumptions of Proposition \ref{mat},  there exists a positive number $t' >0$ such that the pressure function $t \to P(t \F)$ has the following form:
\begin{equation*}
P(t\F)=
\begin{cases}
\infty & \text{ if } t<t';\\
\text{finite, convex and decreasing} & \text{ if } t >t'.
\end{cases}
\end{equation*}
\end{rem}

\section{A Bowen formula} \label{bow}
In this section, we apply the results obtained in order to prove a formula that relates the pressure with the Hausdorff dimension of a geometric construction. This formula generalises previous results by Barreira \cite{b1} to the countable setting.

Let us consider the following geometric construction in the interval. For every $n \in \N$, let $\Delta_n \subset [0,1]$ be a closed interval of length $r_n$. 
Assume that the intervals do not overlap, that is, if $m\neq n$ then $\Delta_m \cap \Delta_n = \emptyset.$ For each  $k \in \N$,  choose again a family of non-overlapping closed 
intervals $\{\Delta_{kn}\}_{n \in \N}$ with $\Delta_{kn}  \subset \Delta_k$. Denote the length of $\Delta_{kn}$ by $r_{kn}$. 
Iterating this procedure, for each interval $\Delta_{i_0 \dots i_{n-1}}$, we obtain a countable family  of non-overlapping closed intervals $\{\Delta_{i_0 \dots i_{n-1} m}\}_{m \in \N}$ with $\Delta_{i_0 \dots i_{n-1} m}  \subset \Delta_{i_0 \dots i_{n-1}}$. Denote by $r_{i_0 \dots i_{n-1} m}$ the length of $\{\Delta_{i_0 \dots i_{n-1} m}\}$. We define the limit set
\begin{equation*}
\K= \bigcap_{n=1}^{\infty} \bigcup_{(i_0 \cdots i_{n-1})}  \Delta_{i_0 \dots i_{n-1}}.
\end{equation*}
This geometric construction can be coded by a full-shift on a countable alphabet $(\Sigma_F, \sigma)$. That is, there exits a homeomorphism 
$\psi :\Sigma_F \to \K$.

Denote by $\phi_n:\Sigma \to \R$ the function defined by
$\phi_n(x)= \log r_{i_0 \dots i_{n-1}}$ if $x \in C_{i_0 \dots i_{n-1}}$ and by $\F=\{\phi_n\}_{n=1}^{\infty}$.

\begin{teo}
Let $\K$ be a geometric construction as above. Assume that there exists $C>0$ such that for every $n,m \in \N$ 
\[  r_{i_0 \dots i_{n-1}} r_{i_{n} \cdots i_{m-1}}e^{-C} \leq r_{i_0 \cdots i_{n+m-1}} \leq r_{i_0 \dots i_{n-1}} r_{i_{n} \cdots i_{m-1}}e^C. \]
Then
\[\dim_H(\K)= \inf \{ t \in \R: P(t \F) \leq 0\} .\]
\end{teo}

\begin{proof}
Let us start with the lower bound. We consider a subset of $K_n \subset\K$ defined by the projection of $\psi$ restricted 
to the compact sub-shift $\Sigma_n \subset \Sigma$, where $\Sigma_n$ is the full-shift on $\{1, 2, \dots , n\}$. Denote by
\[R_n:= \min \{r_1, r_2, \dots, r_n  \}.\]
By the almost-additivity assumption on the radii we obtain that if $i_j \in \{1, 2, \dots , n\}$
then
\begin{equation*}
\frac{r_{i_0 i_2 \dots i_{n}}}{r_{i_0 i_2 \dots i_{n-1}}} \geq R_n e^{-C} := \delta_n
\end{equation*}
Then, a result  proved by Barreira in \cite{b1} (see also \cite[p.35]{b4}) implies that
\[\dim_{H} K_n = t_{K_n},\]
where $t_{K_n} \in \R$ is the unique root of the equation $P(t \F|_{K_n})=0$. By the approximation property of the pressure (see Theorem \ref{approx}), we obtain that
\[ \lim_{n \to \infty} \dim_H K_{n} = \inf \{  t \in \R: P(t \F) \leq 0 \}.\]
Since $K_n \subset \K$, this proves the lower bound.\\

In order to prove the upper bound, we make use of the natural cover. Let  $s \in \R$ be such that $\inf \{ t \in \R: P(t \F) \leq 0\} < s$. Since the pressure of $s \F$ is negative  there exists $L <0$ such that 
\begin{equation*}
\lim_{n \to \infty} \frac{1}{n} \log \sum_{(i_0 \dots i_{n-1}) \in \N^n} (r_{i_0 \dots i_{n-1}})^s = P(s \F) < L < 0.
\end{equation*}
Hence, for sufficiently large values of $n \in \N$ we have that
\begin{equation*}
  \sum_{(i_0 \dots i_{n-1}) \in \N^n} (r_{i_0 \dots i_{n-1}})^s < e^{nL}.
  \end{equation*} 
Since $L<0$ we have that $\lim_{n \to \infty}  e^{nL}=0.$ Note that for each $n \in \N$ the family $\{ \Delta_{i_0 \dots i_{n-1}} : i_0\dots i_{n-1} \in \N^n \}$ is a cover of $\K$. Moreover, as $n$ tends to infinity  the diameters of the sets $\Delta_{i_0 \dots i_{n-1}}$ converge to zero.
This implies that  the $s-$Hausdorff measure of $\K$ is zero. Therefore, we obtain the desired upper bound.
\end{proof}

It should be pointed out that, even in the additive case, it can happen that the equation
$P(t \K)=0$ does not have a root (see the work of Mauldin and Urba\'nski \cite{mu1} and that of Iommi \cite{io} for explicit examples). 

\begin{rem}
If the equation $P(t \F)=0$ has a root and the Gibbs measure corresponding to
$(\dim_H \K) \F$ is an equilibrium measure, then we obtain an almost-additive version of Ledrappier-Young formula (see \cite{ly}) . Indeed, 
\begin{equation*}
P((\dim_H \K) \F)= h(\mu) +\dim_H \K \left( \lim_{n \to \infty} \frac{1}{n} \int \phi_n   \ d \mu \right) =0.  \end{equation*}
Therefore,
\begin{equation*}
\dim_H \K = -\frac{h(\mu)}{ \lim_{n \to \infty} \frac{1}{n} \int \phi_n  \ d \mu}.
\end{equation*}
\end{rem}

\end{document}